\RequirePackage{fix-cm}
\documentclass[envcountsect,envcountsame]{svjour3}
\usepackage{amsmath,amssymb,mathtools,subcaption,algorithm,bm,bbm,xcolor}
\usepackage{booktabs}
\usepackage[colorlinks=true,citecolor=blue,linkcolor=blue,urlcolor=blue,bookmarks,bookmarksopen,bookmarksdepth=2,backref=page]{hyperref}

\usepackage{bookmark}
\bookmarksetup{
	numbered, 
	open,
}

\renewcommand*{\backref}[1]{}
\renewcommand*{\backrefalt}[4]{%
	\ifcase #1 (Not cited.)%
	\or        (Cited on page~#2.)%
	\else      (Cited on pages~#2.)%
	\fi}

\newcommand{\F}{\mathbb{F}}
\newcommand{\V}{\mathbb{V}}

\newcommand{\Z}{\mathbb{Z}}
\renewcommand{\to}{\rightarrow}

\newcommand{\T}{{\rm Tr}}

\def\aa{{\bf a}}
\def\bb{{\bf b}}
\def\cc{{\bf c}}
\def\uu{{\bf u}}

\def\xx{{\bf x}}
\def\yy{{\bf y}}

\def\00{{\bf 0}}
\def\11{{\bf 1}}
\def\+{\oplus}

\DeclareMathOperator{\Tr}{Tr}

\numberwithin{equation}{section}
\newtheorem{fact}[theorem]{Fact}
\newtheorem{construction}[theorem]{Construction}
\smartqed

\begin{document}
\title{Vectorial Negabent Concepts:\ \\ Similarities, Differences, and Generalizations
\thanks{In loving memory of Kai-Uwe Schmidt.}
}

\titlerunning{Vectorial Negabent Concepts}

\author{Nurdag\"{u}l Anbar \and Sadmir Kudin \and Wilfried Meidl \and Enes Pasalic \and Alexandr Polujan}

\authorrunning{N. Anbar, S. Kudin, W. Meidl, E. Pasalic, A. Polujan}

\institute{ 
			Nurdag\"{u}l Anbar \at Sabanc{\i} University, MDBF, Orhanl\i, Tuzla, 34956 \. Istanbul, Turkey \ \\
			\email{nurdagulanbar2@gmail.com}
			\and
            Sadmir Kudin \at 
            University of Primorska, FAMNIT \& IAM, Glagolja\v{s}ka 8, 6000 Koper, Slovenia \\
            \email{sadmir.kudin@iam.upr.si}
            \and
            Wilfried Meidl \at
            Institut f\"ur Mathematik, Alpen-Adria-Universit\"at Klagenfurt, Austria \ \\
            \email{meidlwilfried@gmail.com}  
            \and
            Enes Pasalic \at 
            University of Primorska, FAMNIT \& IAM, Glagolja\v{s}ka 8, 6000 Koper, Slovenia \\
            \email{enes.pasalic6@gmail.com}
            \and 
            Alexandr Polujan \at
              Otto von Guericke University, Universit\"{a}tsplatz 2, 39106, Magdeburg, Germany \\
              \email{alexandr.polujan@ovgu.de}            
}

\date{Received: date / Accepted: date}

\maketitle

\begin{abstract}
In Pasalic et al., IEEE Trans. Inform. Theory 69 (2023), 2702--2712, and in Anbar, Meidl, Cryptogr. Commun. 10 (2018), 235--249, two different vectorial negabent and vectorial bent-negabent concepts are introduced, which leads to seemingly contradictory results. One of the main motivations for this article is to clarify the differences and similarities between these two concepts.  Moreover, the negabent concept is extended to generalized Boolean functions from  \(\mathbb{F}_2^n\) to the cyclic group \(\mathbb{Z}_{2^k}\). It is shown how to obtain nega-\(\mathbb{Z}_{2^k}\)-bent  functions from \(\mathbb{Z}_{2^k}\)-bent functions, or equivalently, corresponding non-splitting  relative difference sets from the splitting relative difference sets. This generalizes the shifting results for Boolean bent and negabent functions. We finally point to constructions of \(\mathbb{Z}_8\)-bent functions employing permutations with the \((\mathcal{A}_m)\) property, and more generally we show that the inverse permutation gives rise to \(\mathbb{Z}_{2^k}\)-bent functions.
% 05B10  	Difference sets (number-theoretic, group-theoretic, etc.) [See also 11B13]
% 06E30  	Boolean functions [See also 94C10]
% 14G50  	Applications to coding theory and cryptography
% 94C30  	Applications of design theory [See also 05Bxx]
\end{abstract}

\section{Introduction}
\thispagestyle{empty}

Let $(A,+_A)$, $(B,+_B)$ be finite abelian groups. A function $f$ from $A$ to $B$ is called a {\it bent function} if
\begin{equation}
	\label{CS}
	|\sum_{x\in A}\chi(x,f(x))| = \sqrt{|A|}
\end{equation}
for every character $\chi$ of $A\times B$ which is non-trivial on $B$. 
Equivalently, $f$ is bent if and only if for all nonzero $a\in A$ the derivative $D_af$
in direction $a$,
\[ D_af(x) = f(x +_A a) -_B f(x) \]
is balanced, i.e, every value of $B$ is taken on the same number, $|A|/|B|$, of times.
This applies if and only if the graph $\mathcal{G}_f = \{(x,f(x))\,:\,x\in A\}$ is a 
splitting relative difference set in $A\times B$ relative to $B$. We refer to \cite{pott} 
for details.

In the classical case, $A = \V_n^{(p)}$ and $B = \V_m^{(p)}$ are elementary abelian $p$-groups, 
i.e., they are vector spaces of dimension $n$ and $m$ respectively over the prime field $\F_p$ 
for some prime $p$. We are here interested solely in the case that $p=2$, i.e., in Boolean and vectorial
Boolean functions, hence we may simply write $\V_n$ for $\V_n^{(2)}$. 
Then the character sum in $(\ref{CS})$, called the {\it Walsh transform} of $f$ at $(a,b)\in\V_m\times\V_n$,
$a\ne 0$, is of the form
\begin{equation*}
	%\label{Walsh} 
	\mathcal{W}_f(a,b) = \sum_{x\in \V_n}(-1)^{\langle a,f(x)\rangle_m + \langle b,x\rangle_n}, 
\end{equation*}
where $\langle \cdot , \cdot \rangle_k$ denotes an inner product in $\V_k$.
In the Boolean case, the Walsh transform reduces to
\begin{equation}
	\label{Walsh1} 
	\mathcal{W}_f(b) = \sum_{x\in \V_n}(-1)^{f(x) + \langle b,x\rangle_n}.
\end{equation}
A function $f\colon\V_n\rightarrow\V_m$ is then \textit{bent} if $|\mathcal{W}_f(a,b)| = 2^{n/2}$ for all nonzero $a\in\V_m$ and $b\in\V_n$.
Clearly, $n$ must then be even, and as it is well-known, $m$ can be at most $n/2$, see 
\cite[Corollary]{nyb}.

Motivated by applications in quantum computing, another class of Boolean functions from 
$\F_2^n$ to $\F_2$, having a flat spectrum with respect to another unitary transform, 
was introduced in \cite{pr}:

For $\cc = (c_1,\ldots,c_n)$, $\xx = (x_1,\ldots,x_n)$ in $\F_2^n$, let $s_2^\cc\colon\F_2^n\rightarrow\F_2$
be the Boolean function
\[ s_2^\cc(\xx) = \sum_{1\le i<j\le n}(c_ix_i)(c_jx_j) \ . \]
Then a unitary transform $\mathcal{U}_f^\cc\colon\F_2^n\rightarrow\mathbb{C}$ is defined by (cf.\cite{gps})
\begin{equation}
	\label{UTF1}
	\mathcal{U}_f^\cc(\bb) = \sum_{\xx\in\F_2^n}(-1)^{f(\xx)+ s_2^\cc(\xx)}i^{\cc\cdot \xx}(-1)^{\bb\cdot \xx}.
\end{equation}
A function $f\colon\F_2^n\rightarrow\F_2$ is called a {\it $\cc$-bent$_4$ function} 
if for every $\bb\in \F_2^n$ we have $|\mathcal{U}_f^\cc(\bb)| = 2^{n/2}$. If $f$ is $\cc$-bent$_4$
for some nonzero $\cc\in \F_2^n$, we call $f$ a {\it bent$_4$ function}. Note that for $\cc=\00=(0, \ldots, 0)$, Equation
$(\ref{UTF1})$ reduces to the Walsh transform $(\ref{Walsh1})$. 

Most attention is given in the literature to $\cc$-bent$_4$ functions $f$ for $\cc = \11 =(1, \ldots,1)$, in which case 
$f$ is called a {\it negabent function}. Basically all results on negabent functions hold for 
$\cc$-bent$_4$ functions for any $\cc\ne \00$ in a similar way.

Univariate versions of $c$-bent$_4$ functions for some $c\in\F_{2^n}$, and negabent functions (for $c=1$) are introduced in \cite{nuwi}, as functions which have a flat spectrum with respect to the transforms
\begin{equation}
	\label{VTF}
	\mathcal{V}_{f}^c(b) = \sum_{x\in\F_{2^n}}(-1)^{f(x)+ \sigma(c,x)}i^{\Tr^n_1(cx)}(-1)^{\Tr^n_1(bx)} \ ,
\end{equation}
where $\Tr^n_1\colon\F_{2^n} \to \F_{2}$ is the \textit{absolute trace}, i.e., $\Tr^n_1(x)=\sum_{i=0}^{n-1} x^{2^i}$, and 
for $c,x\in\F_{2^n}$, $\sigma(c,x)$ is defined as
%\begin{equation}
%\label{sigma}
\[ \sigma(c,x) = \sum_{0\le i<j\le n-1}(cx)^{2^i}(cx)^{2^j}. \]
%\end{equation}
Note that $\sigma(c,x)^2 = \sigma(c,x)$, and hence $\sigma(c,x)$ is a Boolean function.

Similar as for bent functions, bent$_4$ functions can alternatively be defined with a 
modified version of a derivative. A function $f\colon\F_2^n\rightarrow\F_2$ 
($f\colon\F_{2^n}\rightarrow\F_2$) is \textit{$\cc$-bent$_4$ ($c$-bent$_4$)} if
\begin{equation*}
	%\label{negaderi}
	f(\xx) + f(\xx + \aa) + \cc \cdot (\aa \odot \xx) \qquad (f(x) + f(x+a) + \Tr^n_1(cax))
\end{equation*}
is balanced for every nonzero $\aa\in\F_2^n$ ($a\in\F_{2^n}$), 
where for $\aa = (a_1,\ldots,a_n)$ and $\xx = (x_1,\ldots,x_n)$, 
$\aa \odot \xx = (a_1x_1,\ldots,a_nx_n)$.

Bent$_4$ functions also correspond to relative difference sets:\\
The binary operation on the set $\F_2^n\times\F_2$ ($\F_{2^n}\times\F_2$) given by
\begin{align*} 
	& (\xx_1,\yy_1) \star (\xx_2,\yy_2) = (\xx_1+\xx_2,\yy_1+\yy_2+ \cc\cdot(\xx_1 \odot \xx_2)) \\
	& ( (x_1,y_1) \star (x_2,y_2) = (x_1+x_2,y_1+y_2 + \T^n_1(cx_1x_2)) ),
\end{align*}
for some nonzero $\cc\in\F_2^n$ ($c\in\F_{2^n}$), defines a group which is isomorphic to $\F_2^{n-1}\times \Z_4$.
A function $\colon\F_2^n\rightarrow\F_2$ ($\F_{2^n}\rightarrow\F_2$) is \textit{$\cc$-bent$_4$ ($c$-bent$_4$)}
if and only if the graph of $f$ is a (non-splitting) relative difference set in $G$ relative to $2\Z_4$,
see for instance \cite{amp}.

Negabent and bent$_4$ functions have been intensively investigated. We refer to the following pioneering works~\cite{pr,spp} related to these concepts. A fundamental result is that a bent$_4$ function 
(in an even number of variables) is obtained from a bent function with a shift (and vice versa):
\begin{fact}
	\label{fact}
	Let $n$ be an even integer. A function $f \colon \F_2^n\rightarrow\F_2$ ($f \colon \F_{2^n}\rightarrow\F_2$)
	is $\cc$-bent$_4$ ($c$-bent$_4$) if and only if $g(\xx) = f(\xx) + s_2^\cc(\xx)$
	($g(x) = f(x) + \sigma(c,x)$) is a Boolean bent function.
\end{fact}
Also for this reason, research on negabent functions focused on bent-negabent functions, i.e., on functions which are simultaneously bent and negabent.

In \cite{kppp}, a concept for a vectorial version of a negabent function, and in particular of a bent-negabent function is introduced. %We here only state the multivariate version.
For an even integer $n$, a function $F\colon\F_2^n\rightarrow\F_2^m$ ($F\colon\F_{2^n}\rightarrow\F_{2^m}$)
is called a \textit{vectorial negabent function} if all (nonzero) component functions of $F$ are negabent. The function $F$ is called \textit{vectorial bent-negabent} if all (nonzero) components of $F$ are bent-negabent.

Motivated by an analysis of the component functions of modified planar functions in \cite{psz}, a different concept for a vectorial negabent (or bent$_4$) functions is introduced in \cite{nuwi}.

%Motivated by an analysis of the component functions of modified planar functions, a different concept 
%for a vectorial negabent (or bent$_4$) function was introduced in \cite{nuwi} (also remarked in \cite{psz}).
Recall that a function $F$ on $\F_{2^n}$ (w.l.o.g.) is a \textit{modified planar function} if
$F(x) + F(x+a) + ax$ is a permutation for every nonzero $a\in\F_{2^n}$.
As initially pointed out in \cite{psz} (see also \cite{nuwi}), the components of a modified planar function are essentially negabent
(actually bent$_4$) functions. Hence, a modified planar function can be seen as a vectorial version of a
negabent (or bent$_4$) function.

One of the main motivations for this
article is to clarify the differences and similarities between these two concepts, which will be accomplished in Section \ref{sec2}.
In particular, seemingly contradictory results, which arise from these two different concepts
require to be explained: Whereas in \cite{kppp} it is proved that for a vectorial bent-negabent
function from $\V_{2m}$ to $\V_k$, $k$ can be at most $m-1$, in \cite{nuwi} examples of a function from $\F_{2^{2m}}$ to $\F_{2^m}$ are given which --- as stated in \cite{nuwi} ---
``are in some sense vectorial versions of bent-negabent functions".

In Section \ref{sec3}, we introduce generalizations of negabent functions to generalized Boolean functions, i.e., to functions from $\V_n^{(2)}$ to the cyclic group $\Z_{2^k}$. We describe the unitary transform for these functions, and show that similar as for bent and nega-bent functions, one can transform $\Z_{2^k}$-bent functions to nega-$\Z_{2^k}$-bent functions, or equivalently, the corresponding splitting relative difference sets to non-splitting relative difference sets. In Section \ref{sec4}, we investigate constructions of $\mathbb{Z}_8$-bent functions employing permutations with the $(\mathcal{A}_m)$ property, and more generally we show that the inverse permutation gives rise to $\mathbb{Z}_{2^k}$-bent functions.
\section{Vectorial negabent, bent$_4$ and bent-negabent functions}
\label{sec2}
We first turn our attention to the results in \cite{kppp}, where vectorial negabent functions are 
defined as functions $F$ from $\F_2^n$ to $\F_2^k$, for which every (nonzero) component function
is negabent. The paper \cite{kppp} focuses on vectorial bent-negabent functions, hence it is supposed  that $n = 2m$ is even. Particularly, it gives a bound on the dimension of vectorial bent-negabent functions and several construction methods of vectorial bent-negabent functions. It should be emphasized that the design methods in~\cite{kppp} provide instances of vectorial bent-negabent functions that attain the upper bound on the output dimension $k$, that is, $k=m-1$.

\begin{proposition}\cite{kppp}
	\label{kpppprop}
	\begin{itemize}
		\item[-] Let $F \colon \F_2^{2m}\rightarrow\F_2^k$ be a vectorial bent-negabent function. Then, $k$ is at most $m-1$. 
		\item[-] Let $a_1, a_2,\ldots, a_{m-1}$ be elements of $\F_{2^m}$ which are linearly independent
		over $\F_2$, such that their span does not contain the element $1\in\F_{2^m}$. For $i = 1,\ldots,m-1$,
		let $\pi_i$ be the linear permutation $y\rightarrow a_iy$ and $f_i\colon\F_{2^m}\times\F_{2^m}\rightarrow\F_2$
		the Maiorana-McFarland bent function $f_i(x,y) = \T^m_1(x\pi_i(y)) + \rho_i(y)$, where $\rho_i$ is an
		arbitrary function from $\F_{2^m}$ to $\F_2$. Then the function 
		$F \colon \F_{2^m}\times\F_{2^m}\rightarrow\F_2^{m-1}$, $F(x,y) = (f_1(x,y),f_2(x,y),\ldots,f_{m-1}(x,y))$ 
		is affine equivalent to a vectorial bent-negabent function. 
		In fact, if $F^\prime$ denotes the function $F$ represented in multivariate form, i.e., as function 
		from $\F_2^{2m}$ to $\F_2^{m-1}$, then $G(\xx) = F^\prime(\xx A+b)$ is vectorial bent-negabent, where
		$A\in GL(2m,\F_2)$ and $b\in\F_2^{2m}$ are given by $s_2^\11(\xx) = \mu(\xx A + b) + \uu\cdot \xx + e$
		for some $\uu\in\F_2^{2m}$, $e\in\F_2$ and $\mu$ is the quadratic bent function 
		$\mu(\xx) = x_1x_{m+1} + x_2x_{m+2} + \cdots + x_mx_{2m}$.
	\end{itemize}
\end{proposition}

In \cite{zhou}, Zhou introduced the concept of a modified planar function on $\V_n^{(2)}$ to 
express some $(2^n,2^n,2^n,1)$-relative difference sets as a graph of a function.
Relative difference sets with such parameters are particularly interesting, as they give rise to projective planes.

Recall that a function $F$ on $\F_{2^n}$ is modified planar, if for all nonzero $a\in\F_{2^n}$
the modified derivative
\[ F(x+a) + F(x) + ax \]
is a permutation of $\F_{2^n}$. Equivalently, the graph of $F$,
$\mathcal{G}_F = \{(x,F(x))\,:\,x\in\F_{2^n}\}$ is a relative difference set in
$(\F_{2^n}\times \F_{2^n},\star)  \simeq \Z_{4}^n$ where $(x_1,y_1) \star (x_2,y_2) = 
(x_1+x_2,y_1+y_2+x_1x_2)$.
As noted in \cite{psz}, the components of a modified planar function are essentially bent$_4$
functions (which by that time were mainly investigated in the multivariate framework).

As all known classes of modified planar functions have been discovered in univariate representation, for a more detailed analysis of their components the univariate versions of bent$_4$ functions (as given in the introduction) have been introduced in \cite{nuwi}. 
Hence, in the following we will state some main results of \cite{nuwi} in univariate form.
(We remark that the set of univariate bent$_4$ functions is not exactly the set of functions one 
obtains from the multivariate bent$_4$ functions by switching to univariate representation,
see Remark 12 in \cite{nuwi}.)

In Section 4 in \cite{nuwi}, also vectorial versions of bent$_4$ functions are introduced.
As in \cite{nuwi}, we state it in univariate form, and therefore suppose that $k$ divides $n$.
We call a function $F\colon\F_{2^n} \rightarrow\F_{2^k}$ a vectorial bent$_4$ function, if  
\[ F(x+a) + F(x) + \T^n_k(ax) \]
is balanced for every nonzero $a\in\F_{2^n}$, where $\T^n_k$ is the relative trace from $\F_{2^n}$ to $\F_{2^k}$.
In the following proposition, we summarize Proposition 15 and Theorem 8 in \cite{nuwi}.
\begin{proposition}
	\label{nuwires}
	Let $F$ be a function from $\F_{2^n}$ to $\F_{2^k}$. Then the followings are equivalent.
	\begin{itemize}
		\item[(i)] $F$ is a vectorial bent$_4$ function.
		\item[(ii)] For every $u\in\F_{2^n}$ and every nonzero $c\in\F_{2^k}$
		\begin{equation}
			\label{MCV}
			\mathcal{V}_F(c,u) = \sum_{x\in\F_{2^n}}(-1)^{\T^k_1(c^2F(x)) + \T^n_1(ux) + \sigma(c,x)}i^{\T^n_1(cx)}
		\end{equation}
		has absolute value $2^{n/2}$.
		%$\mathcal{V}_F(c,u) = 
		% $\sum_{x\in\F_{2^n}}\chi_{u,c}(x,f(x)) = 
		%\sum_{x\in\F_{2^n}}(-1)^{\T^k_1(c^2F(x)) + \T^n_1(ux) + \sigma(c,x)}i^{\T^n_1(cx)}$
		%has absolute value $2^{n/2}$ for every $u\in\F_{2^n}$ and every nonzero $c\in\F_{2^k}$;
		\item[(iii)] The graph $\mathcal{G}_F$ of $F$ is a relative difference set (relative to
		$\{0\}\times\F_{2^k}$) in the group $(\F_{2^n}\times \F_{2^k}, \star) \simeq 
		\Z_2^{n-k}\times\Z_4^k$, where $(x_1,y_1) \star (x_2,y_2) = (x_1+x_2,y_1+y_2+\T^n_k(x_1x_2))$.
		\item[(iv)] For every nonzero $c\in\F_{2^k}$, the component function $\T^k_1(c^2F(x))$ is $c$-bent$_4$. That is, all of the $p^k-1$ component functions of $F$ are Boolean bent$_4$ functions.
		%All of the $p^k-1$ component functions of $F$ are Boolean bent$_4$ functions, such that for every nonzero $c\in \F_{2^k}$ there exists a component function of $F$ which is $c$-bent$_4$.
	\end{itemize}
\end{proposition}
\begin{remark}
	In \cite[Theorem 8]{nuwi}, Proposition \ref{nuwires}(iv) is shown for modified planar functions, 
	but it is easily verified that it holds more general for vectorial bent$_4$ functions from 
	$\F_{2^n}$ to $\F_{2^k}$.
\end{remark}
\begin{remark}
	Modified planar functions are precisely the vectorial bent$_4$ functions from $\F_{2^n}$ to
	$\F_{2^n}$. Note that $F\colon\F_{2^n}\rightarrow\F_{2^n}$ is a modified planar function if and only
	if $\sum_{x\in\F_{2^n}}(-1)^{\T^n_1(c^2F(x) + ux )+ \sigma(c,x)}i^{\T^n_1(cx)}$ has absolute
	value $2^{n/2}$, for every $u$ and every nonzero $c\in\F_{2^n}$.
	For $k=1$, the definition of a vectorial bent$_4$ function reduces to the definition of a 
	negabent function.
	We also remark that differently from Boolean bent$_4$ functions, in general, vectorial bent$_4$
	functions are not a shift of a (vectorial) bent function. Whereas for a vectorial bent function
	from $\F_{2^n}$ to $\F_{2^k}$, $k$ can be at most $n/2$, we know that there exist vectorial bent$_4$ functions (modified planar functions) on $\F_{2^n}$.
\end{remark}

In \cite[Section 4]{nuwi}, constructions of vectorial Maiorana-McFarland bent functions from $\F_{2^{2m}}$ to $\F_{2^m}$ which are simultaneously vectorial bent$_4$, were proposed. An example
of such a vectorial function $F\colon\F_{2^{2m}}\rightarrow \F_{2^m}$, $m$ odd, is $F(x + \gamma y) = x\pi(y) + g(y)$, where $\gamma$ is a root of $x^2+x+1$, $\pi \in \F_{2^m}[x]$ is a
linearized complete mapping, and $g$ is an arbitrary function on $\F_{2^m}$, see \cite[Corollary 19]{nuwi}. The authors then write
that one {\it can see this function as a vectorial version of a bent-negabent function}.

We here clearly distinguish these two concepts in \cite{nuwi} respectively in \cite{kppp}:
We call a vectorial function $F$ a vectorial negabent
function if every component of $F$ is a negabent function. On the other hand, a function which satisfies
one (hence all) of the equivalent conditions in Proposition \ref{nuwires}, we call a vectorial bent$_4$
function. A vectorial bent-negabent function therefore refers to vectorial functions as dealt with in
Proposition \ref{kpppprop}. A function as given in \cite[Corollary 19]{nuwi} we then call vectorial
bent-bent$_4$.

We close this section with a comparison of these two concepts. \\[.3em]
{\bf Properties of vectorial negabent functions.}
\begin{itemize}
	\item[-] All component functions are negabent functions, i.e., all have a flat spectrum with respect
	to the same unitary transform, 
	$$\mathcal{V}_{g}(b) = \sum_{x\in\F_{2^n}}(-1)^{g(x)+ \sigma(1,x)}i^{\T^n_1(x)}(-1)^{\T_1^n(bx)}$$
	(univariate representation).
	\item[-] $F(x) = (f_1(x),f_2(x),\ldots,f_k(x))$ is vectorial negabent if and only if 
	the graph of every nontrivial linear combination of the functions $f_i$, $1 \le i\le k$, is a  relative difference set in the group $(\F_{2^n}\times\F_2,\star)$ with 
	$(x_1,y_1) \star (x_2,y_2) = (x_1+x_2,y_1+y_2 + \T^n_1(x_1x_2))$.
\end{itemize}
{\bf Properties of vectorial bent$_4$ functions.}
\begin{itemize}
	\item[-] $F$ itself has a flat spectrum with respect to the unitary transform $\mathcal{V}_F$ 
	in $(\ref{MCV})$. The graph of the vectorial function itself, is a relative difference set in a group 
	which is isomorphic to $\Z_2^{n-k}\times\Z_4^k$ (Proposition \ref{nuwires}(ii),(iii)).
	\item[-] Every component function of $F$ is a $c$-bent$_4$ function, hence its spectrum is flat with respect to a unitary transform $\mathcal{V}_f^c$, but $c$ varies with the components.
	The graph of every component function is a relative difference set in a group isomorphic to
	$\Z_2^{n-1}\times\Z_4$, but for every component function the group operation is defined individually.
\end{itemize}

%where $wt(\zz)$ denotes the Hamming weight of $\zz\in\F_2^n$. As a result we can write Equation 
%$(\ref{UTF1})$ also as
%\begin{equation}
%\label{UTF2}
%\mathcal{U}_g^\cc(\uu) = \sum_{\xx\in\F_2^n}(-1)^{g(\xx)+ \uu\cdot \xx}i^{wt(\cc\odot \xx)} \ .
%\end{equation}
%
%\begin{equation*}
%\mathcal{N}_f(u)=\sum_{x \in \F_2^n}(-1)^{f(x)+ u \cdot x} { i}^{wt(x)},
%\end{equation*}
%where $i$ is the imaginary unit, i.e., $i^2=-1$.
%

\section{Negabentness for generalized Boolean functions}
\label{sec3}

The objective in this section is to extend the negabent concept to functions $f$ from $\V_n$ 
to the cyclic group $\Z_{2^k}$, which are called {\it generalized Boolean functions}.
For a generalized Boolean function, the character sum in Equation $(\ref{CS})$ is of the form
\[ \mathcal{H}_f(c,u) = \sum_{x\in\V_n}\zeta_{2^k}^{cf(x)}(-1)^{\langle u,x\rangle_n},
\quad \zeta_{2^k} = e^{2\pi\sqrt{-1}/2^k}. \]
A generalized Boolean function is hence a bent function, called a {\it $\Z_{2^k}$-bent function}, if 
$\mathcal{H}_f(c,u)$ has absolute value $2^{n/2}$ for all $u\in\V_n$ and nonzero $c\in\Z_{2^k}$.
The graph of $f$ is then a splitting relative difference set in $\V_n\times\Z_{2^k}$ relative
to $\{0\}\times\Z_{2^k}$.

Motivated by applications in code division multiple access (CDMA) systems, in \cite{kai},
K.-U. Schmidt initiated research on so-called {\it generalized bent functions (gbent functions)} 
which are functions from $\V_n$ to the cyclic group $\Z_{2^k}$ satisfying the weaker condition 
that $|\mathcal{H}_f(1,u)| = 2^{n/2}$ for all $u\in\V_n$.
Since then, one can observe increasing interest in generalized Boolean functions. For results
on generalized bent functions and on $\Z_{2^k}$-bent functions we refer to 
\cite{HMP2018,wm24,mtqwwf,txqf} and \cite{nuwi22,mepi}, and to Section 7 in the survey paper \cite{survey}.

In the following, we summarize some knowledge on gbent and $\Z_{2^k}$-bent functions
which will be useful in the further.
Let $f \colon \V_n\rightarrow\Z_{2^k}$ be a generalized Boolean function. Then, we can write $f$ uniquely as
\begin{equation*}
	%\label{aaa}
	f(x) = a_0(x) + 2a_1(x) + \cdots + 2^{k-1}a_{k-1}(x),
\end{equation*}
for some Boolean functions $a_i$, $0\le i\le k-1$, from $\V_n$ to $\F_2$.
An efficient characterization of gbent functions is given in \cite{HMP2018} as follows. 
\begin{proposition}
	\label{cor:dualiff}
	Let $f \colon \V_n\rightarrow\Z_{2^k}$ be given as 
	$f(x) = a_0(x) + 2a_1(x) + \cdots + 2^{k-1}a_{k-1}(x)$ for even $n$, and let $\mathcal{A}$ be the affine space of Boolean functions
	\[ \mathcal{A} = a_{k-1} + \langle a_0,a_1,\ldots,a_{k-2}\rangle. \]
	Then, $f$ is gbent if and only if $\mathcal{A}$ is an affine space of  bent functions such that for any three bent functions $g_0,g_1,g_2 \in \mathcal{A}$ we have $(g_0+g_1+g_2)^* = g^*_0+g^*_1+g^*_2$, where $g^*$ denotes the dual of a bent function $g$.
	Equivalently, if $g_0+g_1+g_2 = g_3$, then $g_0^*+g_1^*+g_2^* = g_3^*$.
	%	\begin{itemize}
		%		\item[(i)] If $n$ is even, then $f$ is gbent if and only if $\mathcal{A}$ is an affine space of 
		%		bent functions such that for any three bent functions $g_0,g_1,g_2 \in \mathcal{A}$ we have
		%		$(g_0+g_1+g_2)^* = g^*_0+g^*_1+g^*_2$, where $g^*$ denotes the dual of a bent function $g$.
		%		\item[(ii)]	If $n$ is odd and $f$ is gbent, then $\mathcal{A}$ is an affine space of semibent 
		%		functions, and for every $u\in\V_n$ we have $\mathcal{W}_g(u) = 0$ if and only if
		%		$g\in a_{k-1} \oplus \langle a_0,a_1,\ldots,a_{k-3}\rangle$ or we have $\mathcal{W}_g(u) = 0$ 
		%		if and only if $g\in a_{k-1} \oplus \langle a_0,a_1,\ldots,a_{k-3}\rangle$.
		%	\end{itemize}
\end{proposition}
In \cite{HMP2018}, a similar characterization of gbent functions from $\V_n$ to $\Z_{2^k}$, when $n$  is odd, is given in terms of Boolean semibent functions.

It is quite easily observed that a generalized Boolean function from $\V_n$ to $\Z_{2^k}$
is $\Z_{2^k}$-bent if and only if $2^tf$ is gbent for all $0\le t\le k-1$, see also \cite{HMP2018}.
Note that this also implies that for a $\Z_{2^k}$-bent function 
$f(x) = a_0(x) + 2a_1(x) + \cdots + 2^{k-1}a_{k-1}(x)$, the Boolean function $a_0$
must be bent. Hence for odd $n$, $\Z_{2^k}$ bent functions from $\F_{2^n}$ to $\Z_{2^k}$
do not exist.

There are several approaches to construct gbent functions, also with affine bent spaces $\mathcal{A}$
of a large dimension, see \cite{wm24}. $\Z_{2^k}$-bent functions can be obtained with the very well-known  spread construction. Recently, in connection with the development of the concept of a bent partition, $\Z_{2^k}$-bent function constructions have been presented which are different from the spread  construction, see \cite{akm22a,akm23,akm22,nuwi22}. We recall the definition of a (normal) bent 
partition for the case of characteristic two:
\begin{definition} 
	\label{bepadef}
	Let $K$ be an even integer.
	\begin{itemize}
		\item[(i)] A partition $\Omega =\{A_1,\ldots, A_K\}$ of $\V_n$ is called a \textit{bent partition} of $\V_n$ of 
		depth $K$, if every Boolean function $f$ from $\V_n$ to $\F_2$, of which the support of $f$,
		$supp(f) = \{x\in\V_n\; :\,f(x) = 1\}$ is the union of exactly $K/2$ of the sets $A_j$ in $\Omega$, 
		is a Boolean bent function.
		\item[(ii)] A partition $\Omega =\{U, A_1,\ldots,A_K\}$ of $\V_n$ is called a \textit{normal bent partition} 
		of $\V_n$ of depth $K$, if every function with the following properties is bent:
		\begin{itemize}
			\item[I)] $f$ is constant on $U$ and on the sets $A_1,\ldots, A_K$.
			%Every $c\in\{0,1\}$ has exactly $K/2$ of the sets $A_1,\ldots, A_K$ in its preimage set,
			\item[II)] The support of $f$ contains exactly $K/2$ of the sets $A_j$ in $\Omega$.
		\end{itemize}
	\end{itemize}
\end{definition}
Clearly, every normal bent partition can be made a bent partition by taking the union of $U$ with 
any of the $A_i$. The canonical example of a bent partition is a spread of $\V_n$. In \cite{akm22a},
a large class of bent partitions is introduced, which can be seen as a generalization of a semifield
spread.
\begin{proposition}\cite[Theorem 6]{nuwi22}
	Let $\Omega = \{U,A_0,\ldots,A_{2^k-1}\}$ be a normal bent partition of $V_n$, then the function
	given by $f(x) = j$ if $x\in A_j$ and $f(x) = 0$ (w.l.o.g.) if $x\in U$, is a $\Z_{2^k}$-bent function. \end{proposition}

To extend the negabent concept to generalized Boolean functions, we first specify an 
appropriate generalization of the transform $\mathcal{V}_f$ in $(\ref{VTF})$ for 
generalized Boolean functions.

\subsection{Nega-$\Z_{2^k}$-Hadamard transform}

First recall that a Boolean function $f$ from $\F_{2^n}$ to $\F_2$ is negabent if and
only if one (and hence all) of the following equivalent conditions is satisfied.
\begin{itemize}
	\item[-] The graph $\mathcal{G}_f$ of $f$ forms a $(2^n,2,2^n,2^{n-1})$-relative difference set in
	the group $G=(\mathbb{F}_{2^n}\times\mathbb{F}_{2}, +_\star)$ relative to $\{0\} \times \mathbb{F}_{2}$, where 
	\begin{equation}
		\label{eq:operation}
		(x_1,y_1) +_\star (x_2,y_2) = (x_1+x_2,y_1+y_2+\Tr^n_1(x_1x_2)).
	\end{equation} 
	\item[-] $f(x+y)+f(y)+\mathrm{Tr}^n_1(xy)$ is balanced for all nonzero $y\in \mathbb{F}_{2^n}$.
	\item[-] $f$ has a flat spectrum with respect to the unitary transform $\mathcal{V}_f^1$.	
\end{itemize}

Based in the following lemma, similarly as for the Walsh transform for bent functions, 
the nega-Hadamard transform $\mathcal{V}_f^1$ is obtained from the character group of 
$G$ in which the graph of $f$ is a relative difference set.
\begin{lemma}\cite{APO95}
	\label{alex95}	
	A subset $R$ of cardinality $\kappa$ of a group $G$ of order $\mu \nu$ with a subgroup $N$ of order 
	$\nu$ is a $(\mu, \nu , \kappa, \lambda)$-RDS of $G$ relative to $N$ if and only if for every character 
	$\chi$ of $G$, we have
	\begin{align*}
		|\chi(R)|^2 = \left\{ 
		\begin{array}{ll}
			\kappa^2	& \text{if } \chi=\chi_0, \text{ i.e., }  \chi(g)=1 \text{ for all } g\in G; \\
			\kappa- \lambda \nu &   \text{if } \chi\neq \chi_0 \text{ and }\chi(g)=1 \text{ for all } g\in N ;\\
			\kappa	&  \text{otherwise. } 
		\end{array}
		\right.
	\end{align*}
\end{lemma}

In order to extend the negabent concept to generalized Boolean functions, 
we first specify the appropriate generalization of the group $G$ defined by
(\ref{eq:operation}) above.

Let $G=( \mathbb{F}_{2^n}\times \mathbb{Z}_{2^k}, +_\star)$ be defined by
\begin{align}\label{eq:operation2}
	(x_1,y_1)  +_\star (x_2,y_2)=(x_1+x_2, y_1+y_2+2^{k-1}\mathrm{Tr}^n_1(x_1x_2)).
\end{align}
Apparently, Equation \eqref{eq:operation2} defines a group $G$ with 
identity $(0,0)$, and the inverse of $(x,y)\in G$ is given by 
$-_\star (x,y)=(x,-y+2^{k-1}\mathrm{Tr}^n_1(x))$.
Note that for $k=1$ the operation in Equation \eqref{eq:operation2} reduces to the 
operation given in Equation \eqref{eq:operation}. 
\begin{proposition}
	For $k>1$, the group $G = (\mathbb{F}_{2^n}\times \mathbb{Z}_{2^k}, +_\star)$ with $+_\star$
	given as in Equation \eqref{eq:operation2} is isomorphic to
	$\mathbb{Z}_{2}^{n} \times \mathbb{Z}_{2^{k}}$.
\end{proposition}
\begin{proof}
	As the group $G$ has order $2^{n+k}$,
	we conclude that any element has order $2^\ell$ for some non-negative integer $\ell$.
	Note that we have the following equalities.
	\begin{align}\label{eq:order}
		2^{\ell}(x,y) = \left\{ 
		\begin{array}{ll}
			(x,y)	&  \text{if } \ell=0; \\
			(0, 2y+2^{k-1}\mathrm{Tr}^n_1(x))	&   \text{if } \ell=1;\\
			(0, 2^{\ell}y)	&  \text{if } \ell \geq 2.
		\end{array}
		\right.
	\end{align}
	By Equation \eqref{eq:order}, for $k>1$, we observe that an element has order at most $2^k$. Moreover, the element $(0,1)$ has order $2^k$. Consequently, $\mathbb{Z}_{2^k}$ is a subgroup of $G$. 
	
	Now we count the number of elements of order at most $2$. 
	An element $(x,y)$ has order at most $2$ if and only if $2y+2^{k-1}\mathrm{Tr}^n_1(x)\equiv 0 \mod 2^k$. If $\mathrm{Tr}^n_1(x)=0$, then $2y\equiv 0 \mod 2^k$, which applies if and only if $y\equiv 0 \mod 2^k$ or $y\equiv 2^{k-1} \mod 2^k$. Therefore, there are $2^{n} $ elements $(x,y)$ of order at most $2$ for which $\mathrm{Tr}^n_1(x)=0$. If $\mathrm{Tr}^n_1(x)=1$, then $2y+2^{k-1}\equiv 0 \mod 2^k$, i.e., $y\equiv 2^{k-2} \mod 2^{k-1}$. 
	This applies if and only if $y\equiv 2^{k-2} \mod 2^k$ or $y\equiv 2^{k-1}+2^{k-2} \mod 2^k$. Hence, there are $2^{n} $ elements $(x,y)$ of order at most $2$ for which $\mathrm{Tr}^n_1(x)=1$. We conclude that the
	number of elements of order at most $2$ is $2^{n+1}$, which implies that $G$ contains an isomorphic copy of $\mathbb{Z}_{2}^{n+1}$. Then, the cardinality of $G$ implies that $G=\mathbb{Z}_{2}^{n} \times \mathbb{Z}_{2^{k}}$.
\qed \end{proof}
\begin{remark} If $k=1$, then in $G = (\mathbb{F}_{2^n}\times \mathbb{Z}_{2^k}, +_\star)$  
	there exists an element $(x,y) \in G$ of order $2^{k+1}$. More precisely, the order of $(x,0) \in G$, 
	with $\mathrm{Tr}^n_1(x)=1$, is $4$, see Equation \eqref{eq:order}.
	Consequently, for the case of Boolean negabent functions, i.e., for $k=1$, the situation is different. 
	As it is well-known, for $k=1$, we have $G \cong \mathbb{Z}_2^{n-1} \times \mathbb{Z}_4$.
\end{remark}

In order to apply Lemma \ref{alex95}, we next characterize the character group $\chi_G$ of $G$. 
%For the proof of Proposition \ref{C-group} below, w
We will use the following property of the Boolean function $\sigma(c,x)$.
\begin{lemma}\cite[Lemma 5]{nuwi}
	\label{lem:sigma}
	For $c,x_1,x_2 \in \mathbb{F}_{2^n} $, we have
	\begin{align*}
		\sigma(c,x_1+x_2)=	\sigma(c,x_1)+\sigma(c,x_2)+\mathrm{Tr}^n_1(cx_1)\mathrm{Tr}^n_1(cx_2)+\mathrm{Tr}^n_1(c^2x_1x_2) .
	\end{align*}
\end{lemma}
Let $c_0+c_12+ \cdots +c_{k-1}2^{k-1}$, $c_i\in \mathbb{Z}_{2}$, be the base $2$ representation
of an element $c\in \mathbb{Z}_{2^{k}}$. Then, we define $\sigma(c,x)$ for $c\in\Z_{2^k}$ to be
the Boolean function $\sigma(c_0,x)$, i.e.,
\[ \sigma(c,x) = \left\{\begin{array}{l@{\quad \quad}l}
	\sigma(1,x) & \mbox{if $c$ is odd;} \\
	0 & \mbox{if $c$ is even.}
\end{array}
\right. \]

%For any element $c\in \mathbb{Z}_{2^{k}} $, we can write $c$ in a unique way in the base $2$ as
%$c=c_0+c_12+ \cdots +c_{k-1}2^{k-1}$, where $c_i\in \mathbb{Z}_{2}$ for $i=0, \ldots , k-1$. 
%%Then we define $\sigma (c,x):=\sigma (c_0,x)$. 
%Let $\zeta_{2^k}$ and $i$ be a $2^k$-th and 
%$4$-th roots of unity, respectively.
%
\begin{proposition}
	\label{C-group}
	Let $G=( \mathbb{F}_{2^n}\times \mathbb{Z}_{2^k}, +_\star)$ be the group defined by
	\begin{align*}
		(x_1,y_1)  +_\star (x_2,y_2)=(x_1+x_2, y_1+y_2+2^{k-1}\mathrm{Tr}^n_1(x_1x_2)).
	\end{align*}
	The group $\chi_G$ of characters of $G$ is then $\chi_G=\{ \chi_{u,c} \; : \; u\in\mathbb{F}_{2^n}, c\in \mathbb{Z}_{2^k} \}$, where
	\begin{align}\label{eq:charac}
		\chi_{u,c}(x,y) =(-1)^{\mathrm{Tr}^n_1(ux)+ \sigma (c,x)}\zeta_{2^k}^{cy} i^{\mathrm{Tr}^n_1(c_0x)},
		\quad \zeta_{2^k} = e^{2\pi\sqrt{-1}/2^k},
	\end{align}
	$i$ is a $4$-th root of unity, and $c\equiv c_0 \bmod 2$.
\end{proposition}
\begin{proof} We first show that $\chi_{u,c} \colon  G \mapsto \mathbb{C}$ given by Equation \eqref{eq:charac} is a group homomorphism. For $u\in  \mathbb{F}_{2^n}, c\in \mathbb{Z}_{2^k}$ and $(x_1,y_1), (x_2,y_2)\in \mathbb{F}_{2^n}\times \mathbb{Z}_{2^k} $,
	we have	
	\begin{align}\label{eq:1chi}
		&\chi_{u,c}((x_1,y_1)  +_\star (x_2,y_2)) \\ \nonumber
		=&		\chi_{u,c}(x_1+x_2, y_1+y_2+2^{k-1}\mathrm{Tr}^n_1(x_1x_2)) \\  \nonumber
		=&	(-1)^{\mathrm{Tr}^n_1(u(x_1+x_2))+ \sigma (c,x_1+x_2)}\zeta_{2^k}^{c( y_1+y_2+2^{k-1}\mathrm{Tr}^n_1(x_1x_2))} i^{\mathrm{Tr}^n_1(c_0(x_1+x_2))}.
	\end{align}
	On the other hand, we have
	\begin{align}\label{eq:2chi}
		&\chi_{u,c}(x_1,y_1) \chi_{u,c}(x_2,y_2)  \\ \nonumber
		=&	(-1)^{\mathrm{Tr}^n_1(u(x_1+x_2))+ \sigma (c,x_1)+ \sigma (c,x_2)}\zeta_{2^k}^{c(y_1+y_2)} i^{\mathrm{Tr}^n_1(c_0x_1)+\mathrm{Tr}^n_1(c_0x_2)}.
	\end{align}
	Equations \eqref{eq:1chi} and \eqref{eq:2chi} are equal if and only if 
	\begin{align}\label{eq:3chi}
		\zeta_{2^k}^{c2^{k-1}\mathrm{Tr}^n_1(x_1x_2)}	i^{2\sigma (c,x_1+x_2)  + \mathrm{Tr}^n_1(c_0(x_1+x_2))}
		=i^{ 2\sigma (c,x_1) +2\sigma (c,x_2) + \mathrm{Tr}^n_1(c_0x_1)+ \mathrm{Tr}^n_1(c_0x_2)}.
	\end{align}
	Note that we have 
	\begin{align*}
		\zeta_{2^k}^{c2^{k-1}\mathrm{Tr}^n_1(x_1x_2)}=(-1)^{c_0\mathrm{Tr}^n_1(x_1x_2)}=(-1)^{\mathrm{Tr}^n_1(c_0x_1x_2)}.
	\end{align*}
	Hence, Equation \eqref{eq:3chi} holds if and only if 
	\begin{align*}
		i^{2\sigma (c,x_1+x_2)  +2 \mathrm{Tr}^n_1(c_0x_1x_2) + \mathrm{Tr}^n_1(c_0(x_1+x_2))}
		=i^{ 2\sigma (c,x_1) +2\sigma (c,x_2) + \mathrm{Tr}^n_1(c_0x_1)+ \mathrm{Tr}^n_1(c_0x_2)}.
	\end{align*}
	By Lemma \ref{lem:sigma}, we have 
	\begin{align}\label{eq:4chi}
		\sigma(c,x_1+x_2)=	\sigma(c,x_1)+\sigma(c,x_2)+\mathrm{Tr}^n_1(c_0x_1)\mathrm{Tr}^n_1(c_0x_2)+\mathrm{Tr}^n_1(c_0x_1x_2) .
	\end{align}
	Note that we used the fact that $c_0^2=c_0$, as $c_0\in \mathbb{F}_{2}$.
	Hence, by Equation \eqref{eq:4chi}, Equation \eqref{eq:3chi} holds if and only if
	\begin{align*}
		i^{2\mathrm{Tr}^n_1(c_0x_1) \mathrm{Tr}^n_1(c_0x_2) + \mathrm{Tr}^n_1(c_0(x_1+x_2))}
		=i^{ \mathrm{Tr}^n_1(c_0x_1) +\mathrm{Tr}^n_1(c_0x_2)}.
	\end{align*}
	We observe that 
	\begin{align} \label{eq:5}
		2\mathrm{Tr}^n_1(c_0x_1)\mathrm{Tr}^n_1(c_0x_2)+\mathrm{Tr}^n_1(c_0(x_1+x_2)) \equiv \mathrm{Tr}^n_1(c_0x_1)+ \mathrm{Tr}^n_1(c_0x_2) \mod 4 .
	\end{align}
	This gives the desired equality. 
	
	\noindent  Note that for $(u,c), (v,d)\in \mathbb{F}_{2^n}\times \mathbb{Z}_{2^k}$, we have $\chi_{u,c}=\chi_{v,d}$ if and only if $u=v$ and $c=d$. Hence, Equation \eqref{eq:charac} gives all characters of $G$.
\qed \end{proof}

With Proposition \ref{C-group}, we obtain a unitary transform which generalizes the nega-Hadamard transform $\mathcal{V}_f^1$ for Boolean functions. As in the introduction for negabent functions,
we define this new class of functions first in terms of the respective unitary transform.
\begin{definition}
	\label{negaZbe}
	A function $f$ from $\F_{2^n}$ to $\Z_{2^k}$ is called a \textit{nega-$\mathbb{Z}_{2^k}$-bent function} if
	\begin{align*}
		\mathcal{K}_f(c,u) = \sum_{x\in \mathbb{F}_{2^n} } (-1)^{\mathrm{Tr}^n_1(ux)+ \sigma (c,x)}\zeta_{2^k}^{cf(x)} i^{\mathrm{Tr}^n_1(c_0x)}
	\end{align*}	
	has absolute value $2^{n/2}$, for all $u \in \mathbb{F}_{2^n}$ and nonzero $c\in\mathbb{Z}_{2^k}$, where $c_0\equiv c \mod 2$.
\end{definition}

We now introduce a version of a modified derivative for functions from $\F_{2^n}$ to $\Z_{2^k}$,
which, as we will see, can be used alternatively to define nega-$\mathbb{Z}_{2^k}$-bent functions.
\begin{definition}
	Let $f \colon \mathbb{F}_{2^n}\mapsto \mathbb{Z}_{2^k}$, and $z\in \mathbb{F}_{2^n} $ be a nonzero element.  We define a \textit{modified derivative} of $f$ in the direction $z$ by
	\begin{align}
		\label{modide}
		D_zf(x)=f(x+z)-f(x)+2^{k-1}\mathrm{Tr}^n_1(zx).
	\end{align}
\end{definition}
%
%\begin{definition}
%	Let $f:\mathbb{F}_{2^n}\mapsto \mathbb{Z}_{2^k}$, and $z\in \mathbb{F}_{2^n} $ be a nonzero element. %%	\begin{align*} 
	%	D_zf(x):=f(x+z)-f(x)+2^{k-1}\mathrm{Tr}^n_1(zx).
	%	\end{align*}
%	The function $f$ is called nega $\mathbb{Z}_{2^k}$-bent if for all nonzero $z\in \mathbb{F}_{2^n} $ %the directional derivative $D_zf(x)$ is a balanced function from $\mathbb{F}_{2^n}$ to $ %\mathbb{Z}_{2^k}$.
%\end{definition}
%
\begin{remark}
	For $k=1$, $\mathcal{K}_f(c,u)$ reduces to the nega-Hadamard transform $\mathcal{V}_f^1$,
	the modified derivative in $(\ref{modide})$ reduces to the modified derivative associated
	with Boolean negabent functions.
\end{remark}

Similar as for many of the considered classes of functions, such as bent functions, $\Z_{p^k}$-bent
functions and negabent functions, we now intend to characterize nega-$\mathbb{Z}_{2^k}$-bent functions
in equivalent ways via a flat spectrum with respect to their transforms, as functions for which
a modified version of a derivative is always balanced, or as relative difference sets in
corresponding groups.

Let $G_f=\{ (x,f(x))\, : \, x\in \mathbb{F}_{2^n} \} \subseteq \mathbb{F}_{2^n} \times \mathbb{Z}_{2^k} $ be the graph of a function $f \colon \mathbb{F}_{2^n}\mapsto \mathbb{Z}_{2^k}$. Observe that for 
$(x,f(x)),(y,f(y))\in G_f$, we have 
\begin{align*} 
	(x,f(x))-_\star (y,f(y))=(x+y, f(x)-f(y)+2^{k-1}\mathrm{Tr}^n_1(xy)+2^{k-1}\mathrm{Tr}^n_1(y)).
\end{align*}
Setting $z=x+y$, i.e., $x=y+z$, and observing that
\begin{align*} 
	2^{k-1}\mathrm{Tr}^n_1((y+z)y)+2^{k-1}\mathrm{Tr}^n_1(y)=2^{k-1}\mathrm{Tr}^n_1(zy),
\end{align*}
we have
\begin{align*} 
	(x,f(x))-_\star (y,f(y))=(z,f(y+z)-f(y)+2^{k-1}\mathrm{Tr}^n_1(zy) ).
\end{align*}
We will use the following two lemmas.
%result in \cite[Lemma 9]{nuwi} for the characterization of nega $\mathbb{Z}_{2^k}$-bent functions #
%in terms of their graphs.
\begin{lemma} \label{lem:h}
	\cite[Lemma 9]{nuwi}
	Let $h$ be a complex valued function on $\mathbb{F}_{2^n}$ and let
	$c = c_0 + c_12 + \cdots + c_{k-1}2^{k-1}\in \mathbb{Z}_{2^k}$. Then
	\begin{align*} 
		\Psi_h(u)=\sum_{z\in \mathbb{F}_{2^n} } h(z)
		(-1)^{\mathrm{Tr}^n_1(uz)+ \sigma (c,z)} i^{\mathrm{Tr}^n_1(c_0z)}=h(0),
	\end{align*}	
	for all $u\in \mathbb{F}_{2^n}$, if and only if $h(z)=0$ for all nonzero $z\in \mathbb{F}_{2^n}$.
\end{lemma}
\begin{lemma}
	\label{balanced}
	Let $g$ be a function from $\F_{2^n}$ to $\Z_{2^k}$. Then 
	$\sum_{y\in\F_{2^n}}\zeta_{2^k}^{cg(y)} = 0$,
	for all nonzero $c\in\Z_{2^k}$, if and only if 
	$g$ is balanced.	
\end{lemma}
\begin{proof}
	It is easy to observe that if $g$ is balanced then  $\sum_{y\in\F_{2^n}}\zeta_{2^k}^{cg(y)} = 0$,
	for all nonzero $c\in\Z_{2^k}$. 
	
	We recall that $\{1, \zeta_{2^k}, \ldots , \zeta_{2^k}^{2^{k-1}-1}\}$ is a basis for $\mathbb{Q}(\zeta_{2^k})$ over $\mathbb{Q}$. That is, it is a linearly independent set over $\mathbb{Q}$. Set $B_i=|\{ y\in \F_{2^n} \, :\, g(y)=i\}|$, for $i=0, \ldots,2^k-1 $. Hence, by the fact that $\zeta_{2^k}^{2^{k-1}}=-1$, we have
	\begin{align} \label{eq:k}
		\sum_{y\in\F_{2^n}}\zeta_{2^k}^{cg(y)}&
		=\sum_{i=0}^{2^k-1}B_i\zeta_{2^k}^{ci}
		=\sum_{i=0}^{2^{k-1}-1}(B_i+(-1)^cB_{i+2^{k-1}})\zeta_{2^k}^{ci}.
	\end{align}
	Then the proof of the converse is by induction on $k$. The argument is straightforward for $k=1$. 
	%	
	%We observe that the argument is true for $k=2$, i.e.,   $g:\F_{2^n} \mapsto \Z_{4}$. Then the assumtion by Equation \eqref{eq:k} with $c=1$ implies that 
	%\begin{align}\label{eq:1}
	%	(B_0-B_2)+(B_1-B_3)\zeta_{4}=0.
	%	\end{align}
%	Since $\{1, \zeta_{4}\}$ is linearly independent over $\mathbb{Q}$, we have $B_0=B_2$ and $B_1=B_3$. 
%	Then for $c=2$, by Equation \eqref{eq:k} we have 
%	\begin{align}\label{eq:2}
	%	(B_0+B_2)+(B_1+B_3)\zeta_{4}^2=2B_0-2B_1=0,
	%	\end{align}
%	i.e., $B_0=B_1$.
%	Then by Equations \eqref{eq:1} and \eqref{eq:2} together with $B_0+B_1+B_2+B_3=2^n$, we have $B_0=B_1=B_2=B_3=2^{n-2}$. This implies the balancedness of $g$. \textcolor{red}{We can delete the argument for k=2! It is only for understanding.}
We suppose that the argument is true for any integer $t<k$, where $k\geq 2$. We consider $g \colon \F_{2^n} \mapsto \Z_{2^k}$. For $c=1$, Equation \eqref{eq:k} together with the assumption implies that
\begin{align*}
	\sum_{y\in\F_{2^n}}\zeta_{2^k}^{g(y)}=\sum_{i=0}^{2^{k-1}-1}(B_i-B_{i+2^{k-1}})\zeta_{2^k}^{i} =0.
\end{align*}
Since $\{1, \zeta_{2^k}, \ldots , \zeta_{2^k}^{2^{k-1}-1}\}$ is linearly independent over $\mathbb{Q}$, we have $B_i=B_{i+2^{k-1}}$ for $i=0, \ldots, 2^{k-1}-1$. Hence, for $c=2d$, where $d\in \{1,\ldots,2^{k-1}-1 \}$,  Equation \eqref{eq:k} implies that
\begin{align}\label{eq:3}
	\sum_{y\in\F_{2^n}}\zeta_{2^k}^{2dg(y)}=	\sum_{i=0}^{2^{k-1}-1}2B_i\zeta_{2^k}^{2di} =0.
\end{align}
We now consider $\tilde{g} \colon \F_{2^n} \mapsto \Z_{2^{k-1}}$ defined by $\tilde{g}(y)=g(y) \mod 2^{k-1}$. For $i=0, \ldots,2^{k-1}-1$, we have $\tilde{g}(y)=i$ if and only if $g(y)=i$ or $i+2^{k-1}$. That is,
\begin{align*}
	\tilde{B}_i=|\{ y\in \F_{2^n} \, :\, \tilde{g}(y)=i\}|=B_i+B_{i+2^{k-1}} =2B_i
\end{align*}
for $i=0, \ldots,2^{k-1}-1 $. Then by Equation \eqref{eq:3},
\begin{align*}
	\sum_{y\in\F_{2^n}}\zeta_{2^{k-1}}^{d\tilde{g}(y)}=	\sum_{i=0}^{2^{k-1}-1}2B_i\zeta_{2^{k-1}}^{di} =0
\end{align*}
for any $d=1, \ldots,2^{k-1}-1$. By induction hypothesis, we then have 
$2B_0=2B_1= \cdots =2 B_{2^{k-1}-1}$. By the fact that $B_i=B_{i+2^{k-1}}$ for $i=0, \ldots, 2^{k-1}-1$, 
we obtain the desired equality.
\qed \end{proof}
\begin{theorem} \label{thm:negabent}
Let $f \colon \mathbb{F}_{2^n}\mapsto\mathbb{Z}_{2^k}$. Then the followings are equivalent.
\begin{itemize}
	\item[(i)]  $f$ is nega-$\mathbb{Z}_{2^k}$-bent.
	\item[(ii)] $f(y+z)-f(y)+2^{k-1}\mathrm{Tr}^n_1(zy)$ is balanced, for all nonzero $z\in \mathbb{F}_{2^n}$.
	\item[(iii)] The graph $G_f$ of $f$ forms a $(2^n,2^k,2^n,2^{n-k})$-relative difference set in 
	$G=( \mathbb{F}_{2^n}\times \mathbb{Z}_{2^k}, +_\star)$ relative to $\{0\} \times \mathbb{Z}_{2^k}$.
\end{itemize}
\end{theorem}
\begin{proof}
By the definition of the transform $\mathcal{K}_f$, we have $\chi_{u,c}(G_f) = \mathcal{K}_f(c,u)$.
Observe that $\chi_{u,c}$ is nontrivial on $\{0\}\times\Z_{2^k}$ if and only if $c\ne 0$. The equivalence of
(i) and (iii) follows then from Definition \ref{negaZbe}, Lemma \ref{alex95} and the observation that $\mathcal{K}_f(0,u) = 0$ if $u\ne 0$ and $\mathcal{K}_f(0,0) = 2^n$. \\
To show the equivalence of (i) and (ii), observe that
%		
%	It is enough to show that $(ii)$ holds if and only if $(iii)$ holds. Let $\chi_{u,c}$ be a character %on $\mathbb{F}_{2^n} \times \mathbb{Z}_{2^k}$, which is not trivial on $\{0\} \times \mathbb{Z}_{2^k}$. %Note that this holds if and only if $c\neq 0$. Hence, it is enough to show that $|\chi_{u,c}(G_f)|^2=2^n$ %if $c\neq 0$. Note that we have the following equalities.
%
\begin{align*}
	%	|\chi_{u,c}(G_f)|^2 & =\chi_{u,c}(G_f) \overline{\chi_{u,c}(G_f)}\\
	|\mathcal{K}_f(c,u)|^2 & = \mathcal{K}_f(c,u)\overline{\mathcal{K}_f(c,u)} \\
	& =\left( \sum_{x\in \mathbb{F}_{2^n} } (-1)^{\mathrm{Tr}^n_1(ux)+ \sigma (c,x)}\zeta_{2^k}^{cf(x)} i^{\mathrm{Tr}^n_1(c_0x)}\right)\\
	&\qquad\qquad \qquad \times
	\left( \sum_{y\in \mathbb{F}_{2^n} } (-1)^{\mathrm{Tr}^n_1(uy)+ \sigma (c,y)}\zeta_{2^k}^{-cf(y)} i^{-\mathrm{Tr}^n_1(c_0y)}\right)\\
	&=\sum_{x,y\in \mathbb{F}_{2^n} } (-1)^{\mathrm{Tr}^n_1(u(x+y))+ \sigma (c,x)+\sigma (c,y)}\zeta_{2^k}^{c(f(x)-f(y))} i^{\mathrm{Tr}^n_1(c_0x)-\mathrm{Tr}^n_1(c_0y)} .
\end{align*}
Set $z=x+y$, i.e., $x=y+z$. Then, we have	
\begin{equation*}
\begin{split}
	&|\mathcal{K}_f(c,u)|^2 \\
	=&\sum_{y,z\in \mathbb{F}_{2^n} } (-1)^{\mathrm{Tr}^n_1(uz)+ \sigma (c,y+z)+\sigma (c,y)}\zeta_{2^k}^{c(f(y+z)-f(y))} i^{\mathrm{Tr}^n_1(c_0(y+z))-\mathrm{Tr}^n_1(c_0y)}.
\end{split}	
\end{equation*}
By Equations \eqref{eq:4chi} and \eqref{eq:5}, 
\begin{align*}
	|\mathcal{K}_f(c,u)|^2& =\sum_{y,z\in \mathbb{F}_{2^n} } (-1)^{\mathrm{Tr}^n_1(uz)+\sigma (c,z)}\zeta_{2^k}^{c(f(y+z)-f(y))} i^{2\mathrm{Tr}^n_1(c_0yz)+\mathrm{Tr}^n_1(c_0z)}\\
	& =\sum_{y,z\in \mathbb{F}_{2^n} } (-1)^{\mathrm{Tr}^n_1(uz)+\sigma (c,z)+\mathrm{Tr}^n_1(c_0yz) }\zeta_{2^k}^{c(f(y+z)-f(y))} i^{\mathrm{Tr}^n_1(c_0z)}.
\end{align*}	
Using that	
\begin{align*}
	(-1)^{\mathrm{Tr}^n_1(c_0yz) }=(-1)^{c_0\mathrm{Tr}^n_1(yz) }=\zeta_{2^k}^{2^{k-1}c_0\mathrm{Tr}^n_1(yz)}
	=\zeta_{2^k}^{2^{k-1}c\mathrm{Tr}^n_1(yz)},
\end{align*}
we obtain the following equalities.
\begin{align*}
	|\mathcal{K}_f(c,u)|^2&=\sum_{y,z\in \mathbb{F}_{2^n} } (-1)^{\mathrm{Tr}^n_1(uz)+\sigma (c,z)}\zeta_{2^k}^{c(f(y+z)-f(y)+2^{k-1}\mathrm{Tr}^n_1(yz))} i^{\mathrm{Tr}^n_1(c_0z)}\\
	&=\sum_{z\in \mathbb{F}_{2^n} }(-1)^{\mathrm{Tr}^n_1(uz)+\sigma (c,z)}
	i^{\mathrm{Tr}^n_1(c_0z)}
	\sum_{y\in \mathbb{F}_{2^n} } \zeta_{2^k}^{c(f(y+z)-f(y)+2^{k-1}\mathrm{Tr}^n_1(yz))}. 
\end{align*}
Hence, if $f(y+z)-f(y)+2^{k-1}\mathrm{Tr}^n_1(yz)$ is balanced, for all nonzero $z\in \mathbb{F}_{2^n}$, then $|\mathcal{K}_f(c,u)|^2=2^n$. 

Conversely, suppose that $|\mathcal{K}_f(c,u)|^2=2^n$ for all $u\in \mathbb{F}_{2^n}$. \\
Set $h(z)= \sum_{y\in \mathbb{F}_{2^n} } \zeta_{2^k}^{c(f(y+z)-f(y)+2^{k-1}\mathrm{Tr}^n_1(yz))} $. Note that $h(0)=2^n$. Then, we have 
\begin{align*}
	\Psi_h(u)=|\mathcal{K}_f(c,u)|^2=
	\sum_{z\in \mathbb{F}_{2^n} } h(z)
	(-1)^{\mathrm{Tr}^n_1(uz)+ \sigma (c,z)} i^{\mathrm{Tr}^n_1(c_0z)}=h(0),
\end{align*}
for all $u\in \mathbb{F}_{2^n}$. Hence, by Lemma \ref{lem:h}, we conclude that $h(z)=0$ for all
nonzero $z\in \mathbb{F}_{2^n}$. By Lemma \ref{balanced}, this holds if and only
if $f(y+z)-f(y)+2^{k-1}\mathrm{Tr}^n_1(yz)$ is balanced.
\qed \end{proof}

%\begin{corollary}
%	Let $f:\mathbb{F}_{2^n}\mapsto\mathbb{Z}_{2^k}$. Then $f$ is nega $\mathbb{Z}_{2^k}$-bent if and 
%only if 	
%    \begin{align*}
%	|V_f^c(u)|=\left| \sum_{x\in \mathbb{F}_{2^n} } (-1)^{\mathrm{Tr}^n_1(ux)+ \sigma %   
	%   (c,x)}\zeta_{2^k}^{cf(x)} i^{\mathrm{Tr}^n_1(c_0x)}\right| =2^{n/2}
%	\end{align*}	
%	for all $u  \in \mathbb{F}_{2^n}$, nonzero $c\in\mathbb{Z}_{2^k}$, where $c_0\equiv c \mod 2$.
%\end{corollary}
\begin{remark}
Let $G=( \mathbb{F}_{2^n}\times \mathbb{Z}_{2^k}, +_\star)$ be the group given in Theorem \ref{thm:negabent} and $N$ be the subgroup $N=\{0\} \times \mathbb{Z}_{2^k}$ of $G$. Let $x\in\mathbb{F}_{2^n}$ be an element such that $\mathrm{Tr}^n_1(x)=1$. Then 
$(x,0)$ is an element in the set $\F_{2^n}\times\{0\}$ for which $(x,0)+_\star(x,0)=(0, 2^{k-1})$ 
is in $N$. Hence, as for $k=1$, the forbidden subgroup $N$ of $G$ in Theorem \ref{thm:negabent} is non-splitting also for $k>1$, i.e., if $G$ is represented as 
$( \mathbb{F}_{2^n}\times \mathbb{Z}_{2^k}, +)$ with the conventional component-wise addition, then
the forbidden subgroup is not $\{0\}\times (\Z_{2^k},+)$, but a different subgroup of $G$ isomorphic
to $\Z_{2^k}$.
\end{remark}
\subsection{Nega-gbent and nega-$\Z_{2^k}$-bent functions}
In the previous section, we defined nega-$\Z_{2^k}$-bent functions $f \colon \F_{2^n}\rightarrow\Z_{2^k}$ as the functions for which $\mathcal{K}_f(c,u)$ has absolute value $2^{n/2}$, for all $u\in\V_n$ and nonzero $c\in\Z_{2^k}$.
As for gbent functions, we define a {\it nega-gbent function} as a function 
$f \colon \F_{2^n}\rightarrow\Z_{2^k}$ for which solely $|\mathcal{K}_f(1,u)| = 2^{n/2}$
holds. Note that this formal definition is satisfied at least by the trivial examples. If $g \colon \F_{2^n}\rightarrow\F_2$ is a Boolean negabent function,
then $f=2^{k-1}g$ satisfies $|\mathcal{K}_f(1,u)| = 2^{n/2}$ for all $u\in\F_{2^n}$. 

The main results in this section include a generalization of Fact \ref{fact} on the one-to-one correspondence between bent and negabent functions, to gbent and nega-gbent functions, and further to $\Z_{2^k}$-bent and nega-$\Z_{2^k}$-bent functions.
\begin{theorem}
\label{g-shift}
A function $f(x) = a_0(x) + 2a_1(x) + \cdots + 2^{k-2}a_{k-2}(x) + 2^{k-1}a_{k-1}(x)$ is a nega-gbent function from $\F_{2^n}$ to $\Z_{2^k}$, if and only if 
$g(x) = a_0(x) + 2a_1(x) + \cdots + 2^{k-2}b_{k-1}(x) + 2^{k-1}b_{k-1}(x)$, with 
\[ b_{k-2}(x) = a_{k-2}(x)+\Tr^n_1(x)\quad\mbox{and}\quad b_{k-1}(x) =  a_{k-1}(x)+\sigma(1,x), \] 
is gbent.
\end{theorem}
\begin{proof} Observing that $\zeta_{2^k}^{2^{k-1}} = -1$ and $\zeta_{2^k}^{2^{k-2}} = i$, we can 
write $\mathcal{K}_f(u)=\mathcal{K}_f(1,u)$ as 
\begin{align*}
	\mathcal{K}_f(u) & = \sum_{x\in \mathbb{F}_{2^n} } (-1)^{\mathrm{Tr}^n_1(ux)+ \sigma(1,x)}
	\zeta_{2^k}^{f(x)} i^{\mathrm{Tr}^n_1(x)} \\
	& = \sum_{x\in \mathbb{F}_{2^n} } (-1)^{\mathrm{Tr}^n_1(ux)}
	\zeta_{2^k}^{a_0(x) + 2a_1(x) + \cdots + 2^{k-2}(a_{k-2}(x)+\Tr^n_1(x)) + 2^{k-1}(a_{k-1}(x)+\sigma(1,x))}\\
	&= \mathcal{H}_g(u),
\end{align*} 
which completes the proof.
\qed \end{proof} %\ \\[.5em]
For showing the corresponding result for $\Z_{2^k}$-bent functions, we need some
preparations. 
\begin{lemma}
\label{123}
Let $f(x) = a_0(x) + a_1(x)2 + \cdots + a_{k-2}(x)2^{k-2} + a_{k-1}(x)2^{k-1}$ be a generalized
Boolean function from $\F_{2^n}$ to $\Z_{2^k}$. Then, $f$ is a nega-$\Z_{2^k}$-bent 
function if and only if the following holds:
\begin{itemize}
	\item[(1)] $g(x) = a_0(x) + a_1(x)2 + \cdots + (a_{k-2}(x)+\Tr^n_1(x))2^{k-2} + 
	(a_{k-1}(x)+\sigma(1,x))2^{k-1}$ is gbent,
	\item[(2)] $h(x) = a_0(x) + a_1(x)2 + \cdots + (a_{k-2}(x)+\Tr^n_1(x))2^{k-2} + 
	(a_{k-1}(x)+\sigma(1,x)+\Tr^n_1(x))2^{k-1}$ is gbent,
	\item[(3)] $2^jf(x)$ is gbent for all $j = 1, \ldots, k-1$, or equivalently
	\item[(3$^\prime$)] $\tilde{f}(x) = a_0(x) + a_1(x)2 + \cdots + a_{k-2}(x)2^{k-2}$ is
	$\Z_{2^{k-1}}$-bent.
\end{itemize}
\end{lemma}
\begin{proof}
Per definition, $f$ is nega-$\Z_{2^k}$-bent if and only if $|\mathcal{K}(c,u)|=2^{n/2}$,
for all nonzero $c\in\Z_{2^k}$ and $u\in\F_{2^n}$. \\
For odd $c$, we distinguish between $c\equiv 1\bmod 4$ and $c\equiv 3\bmod 4$. First suppose that
$c\equiv 1\bmod 4$. Then
\begin{align*}
	\mathcal{K}_f(c,u) & = \sum_{x\in\F_{2^n}}(-1)^{\Tr(ux)}\zeta_{2^k}^{cf(x) + 2^{k-1}\sigma(1,x)+
		2^{k-2}\Tr^n_1(x)} \\
	& = \sum_{x\in\F_{2^n}}(-1)^{\Tr(ux)}\zeta_{2^k}^{c(f(x) + 2^{k-1}\sigma(1,x)+2^{k-2}\Tr^n_1(x))}\\
	&= \mathcal{H}_g(c,u)=\mathcal{H}_{cg}(1,u), 
\end{align*}
where $g(x) = a_0(x) + a_1(x)2 + \cdots + (a_{k-2}(x)+\Tr^n_1(x))2^{k-2} + (a_{k-1}(x)+\sigma(1,x))2^{k-1}$. Therefore, $|\mathcal{K}_f(c,u)| = 2^{n/2}$ if and only if $cg$
is gbent, which applies if and only if $g$ is gbent. If $c\equiv 3\bmod 4$, then 
\begin{align*} 
	\mathcal{K}_f(c,u) & = \sum_{x\in\F_{2^n}}(-1)^{\Tr(ux)}\zeta_{2^k}^{cf(x) + 2^{k-1}\sigma(1,x)+
		2^{k-2}\Tr^n_1(x)} \\
	& = \sum_{x\in\F_{2^n}}(-1)^{\Tr(ux)}\zeta_{2^k}^{c(f(x) + 2^{k-1}\sigma(1,x)+32^{k-2}\Tr^n_1(x))} \\
	& = \sum_{x\in\F_{2^n}}(-1)^{\Tr(ux)}\zeta_{2^k}^{c(f(x) + 2^{k-1}(\sigma(1,x)+\Tr^n_1(x))
		+2^{k-2}\Tr^n_1(x))}\\
	& = \mathcal{H}_h(c,u)=\mathcal{H}_{ch}(1,u), 
\end{align*}
where $h(x) = a_0(x) + a_1(x)2 + \cdots + (a_{k-2}(x)+\Tr^n_1(x))2^{k-2} + 
(a_{k-1}(x)+\sigma(1,x)+\Tr^n_1(x))2^{k-1}$. Hence, $|\mathcal{K}_f(c,u)| = 2^{n/2}$ if and only if $ch$ is gbent; equivalently $h$ is gbent. \\
For even $c$, we have $\mathcal{K}_f(c,u) = \mathcal{H}_f(c,u)$, which finishes  the proof. 
\qed \end{proof}
%
%To show the version of Theorem \ref{g-shift} for (nega)-$\Z_{2^k}$-bent functions. We will use 
The following lemma can be inferred from CCZ-equivalence for generalized Boolean functions, see \cite{ceme24}. 
We can also give a simple direct proof.
\begin{lemma}
\label{equiv}
Let $g(x) = a_0(x) + 2a_1(x) + \cdots + 2^{k-1}a_{k-1}(x)$ be a gbent function from $\F_{2^n}$ to 
$\Z_{2^k}$, then $g^\prime(x) = a_0(x) + 2a_1(x) + \cdots + 2^{k-1}(a_{k-1}(x)+\T^n_1(x))$ is also gbent.
\end{lemma}
\begin{proof}
We show that $g^\prime$ satisfies the characterization in Proposition \ref{cor:dualiff},  as $g$ does.
Clearly with $\mathcal{A}_g = a_{k-1} + \langle a_0,\ldots, a_{k-2}\rangle$, also 
$\mathcal{A}_{g^\prime} = a_{k-1} + \Tr^n_1(x) + \langle a_0,\ldots, a_{k-2}\rangle$ is an affine
space of bent functions. Consider three bent functions from $\mathcal{A}_{g^\prime}$, which are then
of the form $g_j(x) + \Tr^n_1(x)$, where $g_j \in\mathcal{A}_g$, $j=0,1,2$.
Using the fact that for a bent function $h \colon \F_{2^n}\rightarrow\F_2$, $(h(x)+\Tr^n_1(\alpha x))^* = h^*(x+\alpha)$,
we have the following equalities.
\begin{align*}
	& (g_0(x)+\Tr^n_1(x)+g_1(x)+\Tr^n_1(x)+g_2(x)+\Tr^n_1(x))^*\\
	& \qquad \qquad =(g_0(x)+g_1(x)+g_2(x)+\Tr^n_1(x))^*\\
	& \qquad \qquad =(g_0+g_1+g_2)^*(x+1)\\
	& \qquad \qquad =  g_0^*(x+1)+g_1^*(x+1)+g_2^*(x+1) \\
	& \qquad \qquad =(g_0(x)+\Tr^n_1(\alpha x))^*+(g_1(x)+\Tr^n_1(\alpha x))^*+(g_2(x)+\Tr^n_1(\alpha x))^*,
\end{align*}
which completes the proof.
\qed \end{proof}
%\ \\[.5em]
%
With Lemma \ref{123} and Lemma \ref{equiv}, we obtain the version of Theorem \ref{g-shift} for
nega-$\Z_{2^k}$-bent functions.
\begin{theorem}
\label{k-shift}
The function 
\[ f(x) = a_0(x) + 2a_1(x) + \cdots + 2^{k-3}a_{k-3}(x) + 2^{k-2}a_{k-2}(x) + 2^{k-1}a_{k-1}(x), \]
from $\F_{2^n}$ to $\Z_{2^k}$, is nega-$\Z_{2^k}$-bent if and only if 
\[ g(x) = a_0(x) + 2a_1(x) + \cdots + 2^{k-3}a_{k-3}(x) + 2^{k-2}b_{k-2}(x) + 2^{k-1}b_{k-1}(x), \]
with $b_{k-2}(x) = a_{k-2}(x)+\Tr^n_1(x)$ and $b_{k-1}(x) = a_{k-1}(x)+\sigma(1,x)$, is 
$\Z_{2^k}$-bent.
\end{theorem}
\begin{proof}
First suppose that $g$ is a $\Z_{2^k}$-bent function. Then, $g$ is gbent and Condition (1) in 
Lemma \ref{123} is satisfied. %(by Theorem \ref{g-shift}, $f$ is then nega-gbent).
By Lemma \ref{equiv}, with $g$, also
\[ h(x) = a_0(x) + a_1(x)2 + \cdots + (a_{k-2}(x)+\Tr^n_1(x))2^{k-2} + 
(a_{k-1}(x)+\sigma(1,x)+\Tr^n_1(x))2^{k-1}, \]
is gbent. Therefore also (2) in Lemma \ref{123} holds. It remains to show that (3$^\prime$),
\[ \tilde{f}(x) = a_0(x) + 2a_1(x) + \cdots + 2^{k-3}a_{k-3}(x) + 2^{k-2}a_{k-2}(x) \]
is a $\Z_{2^{k-1}}$-bent function. This follows from the assumption that 
\[ \tilde{g}(x) = a_0(x) + 2a_1(x) + \cdots + 2^{k-3}a_{k-3}(x) + 2^{k-2}(a_{k-2}(x)+\Tr^n_1(x)) \]
is $\Z_{2^{k-1}}$-bent, and Lemma \ref{equiv}. \\
Conversely, if $f$ is nega-$\Z_{2^k}$-bent, then $a_0(x) + 2a_1(x) + \cdots + 2^{k-3}a_{k-3}(x) + 2^{k-2}a_{k-2}(x)$ is $\Z_{2^{k-1}}$-bent, hence
$a_0(x) + 2a_1(x) + \cdots + 2^{k-3}a_{k-3}(x) + 2^{k-2}(a_{k-2}(x)+\Tr^n_1(x))$ is $\Z_{2^{k-1}}$-bent, by Lemma \ref{equiv}. By Theorem \ref{g-shift}, since $f$ is nega-gbent, $g$ is gbent, and therefore $\Z_{2^k}$-bent.
\qed \end{proof}
A large variety of $\Z_{2^k}$-bent functions from $\V_n$, $n=2m$, to $\Z_{2^k}$ can be obtained 
from bent partitions like spreads or generalized semifield spreads, see e.g.  \cite[Theorem 6]{nuwi22} and 
\cite[Theorem 1]{akm23}. Using Theorem \ref{k-shift}, we then obtain a large variety of corresponding 
nega-$\Z_{2^k}$-bent functions. Like bent functions, $\Z_{2^k}$-bent functions from $\V_n$ to $\Z_{2^k}$ 
cannot exist when $n$ is odd. That is, Condition (3$^\prime$) in Lemma \ref{123} cannot be satisfied 
when $n$ is odd and $k >1$. Hence, we have the following corollary.
\begin{corollary}
For $k >1$ and odd integers $n$, nega-$\Z_{2^k}$-bent functions from $\F_{2^n}$ to $\Z_{2^k}$ do not
exist.
\end{corollary}
As there are Boolean negabent functions $f \colon \V_n\rightarrow\F_2$ also for odd $n$, this is different
for $k=1$.

\section{Gbent and $\mathbb{Z}_{2^k}$-bent functions from permutations with the $(\mathcal{A}_m)$ property}
\label{sec4}

As remarked in the previous section, a huge quantity of $\Z_{2^k}$-bent functions (and hence by Theorem \ref{k-shift} of nega-$\Z_{2^k}$-bent functions) can be obtained from a bent partition,
such as a generalized semifield spread.
At the same time, a large variety of bent partitions is known. Apart from the generalized semifield
spreads, in \cite{wfw} and in the articles \cite{aakm23,jedli} on the strongly related concept of of Latin square partial difference sets packings (LP-packings), some secondary constructions of bent partitions are introduced.

Conversely, there exist $\Z_{2^k}$-bent functions, which do not come from the known constructions,
and not even from a bent partition. One example is given in \cite[Remark 7]{akm22}:
For an integer $e$ with $\gcd(2^m-1,e) = 1$ and $c_1,c_2,c_3\in\F_{2^m}$ such that 
$c_1^{-e}+c_2^{-e}+c_3^{-e} = (c_1+c_2+c_3)^{-e}$, the function $f$ from $\F_{2^m}\times\F_{2^m}$ to 
$\Z_8$ given by
\begin{equation}
\label{CEx} 
f(x,y) = \Tr^m_1((c_1^{-e}+c_2^{-e})x^ey) + 2\Tr^m_1((c_1^{-e}+c_3^{-e})x^ey) + 
4\Tr^m_1(c_1^{-e}x^ey)
\end{equation}
is a $\Z_8$-bent function. For some choices of $e,c_1,c_2,c_3$, the preimage set partition of $f$ 
is not a bent partition of $\Z_{2^m}\times\Z_{2^m}$. In particular, these functions cannot come from a generalized semifield spread or the above mentioned secondary construction.

The idea behind the construction of $(\ref{CEx})$ is Proposition \ref{cor:dualiff}
together with the observation that $f \colon \V_n\rightarrow\Z_{2^k}$ is $\Z_{2^k}$-bent if and only if $2^tf$ is gbent, for all $0\le t\le k-1$.
%
%\begin{proposition}
%	\label{cor:dualiff}
%	A function $f:\V_n\rightarrow\Z_{2^k}$, $n$ even, given as $f(x) = a_0(x) + 2a_1(x) + \cdots + %2^{k-1}a_{k-1}(x)$ is gbent if and only if
%	\[ \mathcal{A} = a_{k-1} \oplus \langle a_0,a_1,\ldots,a_{k-2}\rangle \]
%	is an affine vector space of bent functions such that for any $g_0,g_1,g_2,g_3 \in \mathcal{A}$ with $g_0+g_1+g_2+g_3 = 0$ we have
%	$g_0^*+g_1^*+g_2^*+g_3^* = 0$. Equivalently, if $g_3 = g_0+g_1+g_2$, then $g^*_3 = %g^*_0+g^*_1+g^*_2$.
%\end{proposition}
%
With the conditions on $e,c_i, i = 1,2,3$ it is guaranteed that the Maiorana-McFarland bent 
functions involved in $(\ref{CEx})$ satisfy the conditions in Proposition \ref{cor:dualiff}.

More generally, bent functions $g_0,g_1,g_2,g_3$ with $g_3=g_0+g_1+g_2$ and $g_3^*=g_0^*+g_1^*+g_2^*$
can be obtained using permutations $\pi_1,\pi_2,\pi_3,\pi_4$ of $\F_{2^m}$ satisfying the $(\mathcal{A}_m)$
property defined as below, via the corresponding Maiorana-McFarland bent functions on $\F_{2^m}\times\F_{2^m}$ of the form $f_i(x,y)=\Tr^m_1(x \pi_i(y))+h_i(y)$. The only condition the ingredients of the Maiorana-McFarland functions $f_i$ have to satisfy is 	$h_1(\pi_1^{-1}(y))+h_2(\pi_2^{-1}(y))+h_3(\pi_3^{-1}(y))+h_4(\pi_4^{-1}(y))=0$.
\begin{definition}\cite{mes14}
Let $\pi_1,\pi_2,\pi_3$ be three permutations of $\F_{2^m}$. We say that $\pi_1,\pi_2,\pi_3$ satisfy \textit{the $(\mathcal{A}_m)$ property} if
\begin{enumerate}
	\item $\pi_4=\pi_1+\pi_2 + \pi_3$  is  a permutation and
	\item $\pi^{-1}_4=\pi_1^{-1} + \pi_2^{-1} + \pi_3^{-1} $.
\end{enumerate} 
\end{definition}

The search for more $\Z_8$-bent functions (hence nega-$\Z_8$-bent functions), which do not come from a generalized semifield spread or even not from a bent partition motivates us to investigate permutations  that satisfy the $(\mathcal{A}_m)$ property. Such permutations could be constructed with the help of permutation monomials using the following result.

\begin{theorem}\cite{Mesnager2015} \label{th:permutA_m2nd class}
Let $m \geq 3$ be an integer and $d^2 \equiv 1 \mod{2^m-1}$. Let $\pi_i$ be three permutations of $\F_{2^m}$ defined by $\pi_i(y)=\alpha_i y^d$, for $i=1,2,3$, where $\alpha_i \in \F_{2^m}^*$ are pairwise distinct  elements such that $\alpha_i^{d+1}=1$ and $\alpha_4^{d+1}=1$ where $\alpha_4=\alpha_1+\alpha_2 + \alpha_3$. Then, the permutations $\pi_i$ satisfy the property ($\mathcal{A}_m$) and furthermore $\pi_i$ are involutions. % as well as $\pi_4=\pi_1+\pi_2+\pi_3$.
\end{theorem}
\subsection{Constructions of gbent and $\mathbb{Z}_{2^3}$-bent functions}
The following example illustrates how one can construct gbent functions and $\Z_8$-bent functions of small dimension using the above described results.
\begin{example}
\label{4.1}
Let $m=3$, and let the multiplicative group of $\F_{2^3}$ be given by $\F_{2^3}^*=\langle a \rangle$, where $a^3+a+1=0$. Let $d=2^m-2=6$, which satisfies $d^2 \equiv 1 \mod 7$. Define $\alpha_1=a, \alpha_2=a^4,\alpha_3=a^6$ and $\alpha_4=\alpha_1+\alpha_2+\alpha_3=1$.  By Theorem~\ref{th:permutA_m2nd class}, the mappings $\pi_i(y)=\alpha_i y^d$, for $i=1,2,3$, are involutions, as well as $\pi_4=\pi_1+\pi_2+\pi_3$. Define the Boolean functions 
$f_i(x,y) = \Tr^m_1(x\pi_i(y))+h_i(y)$, with	$h_i(y)=0$ for all $i=1,2,3,4$.
Since the Boolean functions $h_i$ trivially satisfy the condition 
\begin{equation*}
	h_1(\pi_1^{-1}(y))+h_2(\pi_2^{-1}(y))+h_3(\pi_3^{-1}(y))+h_4(\pi_4^{-1}(y))=0
\end{equation*}
and  $\pi_4=\pi_1+\pi_2+\pi_3$, we have that $f_1+f_2+f_3+f_4 = 0$ and $f_1^*+f_2^*+f_3^* +f_4^*= 0$. Now, from the functions $f_i$ we construct an affine space of bent functions, which, in turn, defines a gbent function. First, we observe that the set
\begin{equation*} %\label{eq: affine space}
	\{ \alpha_1=a, \alpha_2=a^4,\alpha_3=a^6,\alpha_4=1 \}=a+\{0, a^2,a^5,a^3  \}=a+\langle a^2,a^5\rangle
\end{equation*}
is in fact an affine space. Now set
\begin{equation*}
	\begin{split}
		a_2(x,y)&=f_1(x,y)=\Tr^m_1(x(ay)),\\
		a_0(x,y)&=f_2(x,y)+f_1(x,y)=\Tr^m_1(x(a^2y)),\\
		a_1(x,y)&=f_3(x,y)+f_1(x,y)=\Tr^m_1(x(a^5y)).\\
	\end{split}
\end{equation*}
Then, the function $f(x,y) = a_2(x,y) + 2a_0(x,y) + 4a_{1}(x,y)$ is gbent, since for the only possible choice of functions $g_0,g_1,g_2,g_3 \in \mathcal{A}= a_{2} + \langle a_0,a_1\rangle $ such that $g_0+g_1+g_2+g_3 = 0$, namely $g_0=a_2=f_1, g_1=a_0 + a_2=f_2, g_2=a_1 + a_2=f_3, g_3=a_0 + 
a_1 + a_2=f_4$, we have $g_0^*+g_1^*+g_2^*+g_3^* = 0$. 
\end{example}

Based on this example, we propose an infinite family of gbent functions arising from the permutations of $\F_{2^m}$ with the $(\mathcal{A}_m)$ property.

\begin{construction}\label{con: hi trivial}
Let $m\in \mathbb{N}$, and let $a$ be a primitive element of $\F_{2^m}$. Let $d=2^m-2$, which satisfies $d^2 \equiv 1 \mod 2^m-1$. Define $$\alpha_1=1, \; \alpha_2=a,\;\alpha_3=a^d\quad\mbox{and}\quad\alpha_4=\alpha_1+\alpha_2+\alpha_3.$$
Since $\alpha_i \in \F_{2^m}^*$ are pairwise distinct elements satisfying $\alpha_i^{d+1}=1$ and $\alpha_4^{d+1}=1$, the permutations $\pi_i$ satisfy the property ($\mathcal{A}_m$) and $\pi_i$ are involutions by Theorem~\ref{th:permutA_m2nd class}. Define the Boolean functions 
$f_i(x,y) = \Tr^m_1(x\pi_i(y))+h_i(y)$ with	$h_i(y)=0$ for all $i=1,2,3,4$.
Since the Boolean functions $h_i$  satisfy the condition 
\begin{equation*}
	h_1(\pi_1^{-1}(y))+h_2(\pi_2^{-1}(y))+h_3(\pi_3^{-1}(y))+h_4(\pi_4^{-1}(y))=0
\end{equation*}
and  $\pi_4=\pi_1+\pi_2+\pi_3$, we have that $f_1+f_2+f_3+f_4 = 0$ and $f_1^*+f_2^*+f_3^* +f_4^*= 0$. Now, from the functions $f_i$ we construct  an affine space of bent functions, which, in turn, defines a gbent function. Since  $\alpha_1,\alpha_2,\alpha_3,\alpha_4$ are four different elements satisfying  $\alpha_1+\alpha_2+\alpha_3+\alpha_4=0$, the set $\{\alpha_1,\alpha_2,\alpha_3,\alpha_4\}$ is an affine subspace, which can be written in the form
\begin{equation*}  %\label{eq: affine space general}
	\begin{split}
		\{ \alpha_1=1, \alpha_2=a,\alpha_3=a^d,\alpha_4=1+a+a^d \}&=1+\{0, 1+a,1+a^d,a+a^d  \}\\
		&=1+\langle 1+a,1+a^d\rangle .
	\end{split}
\end{equation*}
Now, define the functions $a_i$ in the following way:
\begin{equation*}
	\begin{split}
		a_2(x,y)&=f_1(x,y)=\Tr^m_1(x(1\cdot y)),\\
		a_0(x,y)&=f_2(x,y)+f_1(x,y)=\Tr^m_1(x((1+a)\cdot y)),\\
		a_1(x,y)&=f_3(x,y)+f_1(x,y)=\Tr^m_1(x((1+a^d)\cdot y)).\\
	\end{split}
\end{equation*}
Then, the function $f(x,y) = a_2(x,y) + 2a_0(x,y) + 4a_{1}(x,y)$ is gbent, since  for the only possible choice of functions $g_0,g_1,g_2,g_3 \in \mathcal{A}= a_{2} + \langle a_0,a_1\rangle $ such that $g_0+g_1+g_2+g_3 = 0$, namely $g_0=a_2=f_1, g_1=a_0 + a_2=f_2, g_2=a_1 + a_2=f_3, g_3=a_0 + 
a_1 + a_2=f_4$, we have $g_0^*+g_1^*+g_2^*+g_3^* = 0$. 
\end{construction}

In the following statement, we indicate, that apart from the trivial choice of the functions $h_1=h_2=h_3=h_4=0$, there exists a plenty of options of selecting different Boolean functions $h_i$. For the sake of a more clear presentation, we begin with a very simple case when $\pi_i(y)=\alpha_iy^d$ and $h_i(y)=\Tr^m_1(\alpha_i y^k)$, where $k=d$. However, as we show later, one can take $k\neq d$ and 
in fact use non-monomials $h_i$. The next result is based on the ideas used in~\cite[Proposition 
4.2]{ppkz}.
\begin{proposition}\label{prop: nontrivial hi correct}
Let $m \ge 3$ and $\pi_i(y)=\alpha_iy^d$ for $i=1,2,3,4$ be involutions of $\F_{2^m}$ defined as in Theorem~\ref{th:permutA_m2nd class}. Define the Boolean functions $h_i$ on $\F_{2^m}$	for $i=1,2,3,4$  as follows:
\begin{equation*}
	h_i(y)=\Tr^m_1(\beta_i y^k)\quad\mbox{for }i=1,2,3,4,
\end{equation*}
where $k=d$ and the elements $\beta_i\in\F_{2^m}^*$ are given by 
\begin{equation*}
	\beta_1=\alpha_1,\; \beta_2=\alpha_2,\; \beta_3=\alpha_3,\; \beta_4=\alpha_4.
\end{equation*}
Then, the Maiorana-McFarland bent functions $f_i(x, y)=\Tr^m_1\left(x \pi_i(y)\right)+h_i(y)$, for 
$i \in\{1,2,3,4\}$ and $x, y \in$ $\mathbb{F}_{2^{m}}$, satisfy
\begin{itemize}
	\item[i)] $f_1(x,y)+f_2(x,y)+f_3(x,y)+f_4(x,y)=0$,
	\item[ii)] $h_1(\pi_1^{-1}(y))+h_2(\pi_2^{-1}(y))+h_3(\pi_3^{-1}(y))+h_4(\pi_4^{-1}(y))=0$.
\end{itemize}
\end{proposition}
\begin{proof}
The statement $i)$ follows immediately from the fact that $f_4$ is defined by $ f_4(x,y)=f_1(x,y)+f_2(x,y)+f_3(x,y)$. Now, we show that the statement $ii)$ holds as well. Since 
all permutations $\pi_i(y)=\alpha_i y^d$ of $\F_{2^m}$ are involutions, we have 
\begin{equation*}
	\begin{split}
		\sum\limits_{i=1}^4 	h_i(\pi_i^{-1}(y))&=\Tr^m_1\left(\beta_1\alpha_1^ky^{kd}+\beta_2\alpha_2^ky^{kd}+\beta_3\alpha_3^ky^{kd}+\beta_4\alpha_4^ky^{kd}\right)\\
		&=\Tr^m_1\left(\alpha_1\alpha_1^dy^{d^2}+\alpha_2\alpha_2^dy^{d^2}+\alpha_3\alpha_3^dy^{d^2}+\alpha_4\alpha_4^dy^{d^2}\right)\\
		&=\Tr^m_1\left(\left(\alpha_1^{d+1}+\alpha_2^{d+1}+\alpha_3^{d+1}+\alpha_4^{d+1}\right)y^{d^2}\right)=0,\\
	\end{split}
\end{equation*}
since  $\alpha_i^{d+1}=1$ holds for all $i=1,2,3,4$.
\qed \end{proof}

Using non-trivial selection of the functions $h_i$, we can now construct gbent and $\Z_8$-bent functions, following Construction~\ref{con: hi trivial}.

\begin{construction}\label{con: hi non-trivial}
Let Maiorana-McFarland bent functions $f_i$ on $\F_{2^m}\times\F_{2^m}$ be defined as in Proposition~\ref{prop: nontrivial hi correct}.  Define the Boolean bent functions $a_i$ on $\F_{2^m}\times\F_{2^m}$ in the following way:
\begin{equation*}
	\begin{split}
		a_2(x,y)&=f_1(x,y),\\
		a_0(x,y)&=f_2(x,y)+f_1(x,y),\\
		a_1(x,y)&=f_3(x,y)+f_1(x,y).\\
	\end{split}
\end{equation*}
Then, the function $f\colon\F_{2^m}\times\F_{2^m}\to\Z_8$, given by $$f(x,y) = a_2(x,y) + 2a_0(x,y) + 4a_{1}(x,y),$$ is gbent, since 
the only functions $g_0,g_1,g_2,g_3 \in \mathcal{A}= a_{2} + \langle a_0,a_1\rangle $ with $g_0+g_1+g_2+g_3 = 0$, are exactly $g_0=a_2=f_1, \; g_1=a_0 + a_2=f_2, \; g_2=a_1 + a_2=f_3, \; g_3=a_0 + 
a_1 + a_2=f_4$. Moreover, we additionally have $g_0^*+g_1^*+g_2^*+g_3^* = 0$. Since $f_i+f_j$ is bent for all $1 \le i < j \le 3$, we have that $f$ is $\Z_8$-bent.
\end{construction}

\begin{example}
\label{Ex4.7}
Let $m=3$. Take $d=2^m-2=6$, which satisfies $d^2 \equiv 1  \mod 2^m-1$. Let the multiplicative group of $\F_{2^m}$ be given by $\F_{2^m}^*=\langle a \rangle$, where $a^6 + a^4 + a^3 + a + 1 = 0$.
Define $\alpha_1=a,\; \alpha_2=a^4,\; \alpha_3=a^6,\; \alpha_4=\alpha_1+\alpha_2+\alpha_3=1$. Following Proposition~\ref{prop: nontrivial hi correct}, define the Maiorana-McFarland Boolean bent functions $f_i\colon\mathbb{F}_{2^{m}}\times \mathbb{F}_{2^{m}}\to\F_2$ by $f_i(x, y) 
=\Tr^m_1\left(x \pi_i(y)\right)+h_i(y)$, where $\pi_i(y)=\alpha_iy^d$ and $h_i(y)=
\Tr^m_1(\alpha_i y^d)$, for $i=1,2,3,4$. Then, let  $f\colon\F_{2^m}\times\F_{2^m}\to\Z_8$  be
given by 
\begin{equation}
	\begin{split}
		f(x,y) &= f_1(x,y) + 2(f_2(x,y)+f_1(x,y)) + 4(f_3(x,y)+f_1(x,y)) \\
		&=\Tr^m_1(x\alpha_1y^d)+\Tr^m_1(\alpha_1y^d) + 2(\Tr^m_1(x(\alpha_2+\alpha_1)y^d)+\Tr^m_1((\alpha_2+\alpha_1)y^d)) \\
		&+4(\Tr^m_1(x(\alpha_3+\alpha_1)y^d)+\Tr^m_1((\alpha_3+\alpha_1)y^d)) \\
		&=\Tr^m_1(x ay^6+ay^6) + 2\Tr^m_1(xa^2y^6+a^2y^6)+4\Tr^m_1(xa^5y^6+a^5y^6).
	\end{split}
\end{equation}
Since 
$ \mathcal{H}_f(c,u)\in\left\{\pm 8,\pm 8 i,\pm (4+4 i) \sqrt{2},\pm(4-4 i) \sqrt{2}\right\}$,
we have $|\mathcal{H}_f(c,u)|=8=2^m$, for all $u  \in \mathbb{F}_{2^m}\times \mathbb{F}_{2^m}$ and all nonzero $c\in\mathbb{Z}_{2^3}$. Hence, $f$ is $\Z_8$-bent.
\end{example}
Using Magma~\cite{magma}, we checked the properties of the preimage set partition of the $\Z_8$-bent function in
Example \ref{Ex4.7}. We confirmed that it is a bent partition, which we obtained employing
permutations satisfying the $(\mathcal{A}_m)$ property.
\begin{remark}
\emph{1)} In Proposition~\ref{prop: nontrivial hi correct}, we considered  $\pi_i(y)=\alpha_iy^d$  involutions of $\F_{2^m}$ from Theorem~\ref{th:permutA_m2nd class} and Boolean functions
$h_i(y)=\Tr^m_1(\alpha_i y^k)$ with $k=d$. However, one can take $k\neq d$, provided that $\alpha_1^{k+1}=\alpha_2^{k+1}=\alpha_3^{k+1}=\alpha_4^{k+1}=1$. Note that it is always possible to find such $\alpha_i$, provided that $\F_{2^m}$ is large enough in the sense that it contains a subfield $\F_{2^l}$ with $l\ge 3$. Then one can take $\alpha_i\in\F_{2^l}$ and take $k=2^l-2$, i.e., the inversion in $\F_{2^l}$. Then, clearly,  $\alpha_1^{k+1}=\alpha_2^{k+1}=\alpha_3^{k+1}=\alpha_4^{k+1}=1$ and $\alpha_1^{d+1}=\alpha_2^{d+1}=\alpha_3^{d+1}=\alpha_4^{d+1}=1$, where $d=2^m-2$. Taking a chain of subfields $\F_{2^{l_1}}\subset \F_{2^{l_2}}\subset\cdots\subset \F_{2^{l_s}}=\F_{2^{m}}$, and corresponding $k_i=2^{l_i}-2$, one can use the functions $h_i(y)=\Tr^m_1\left( \alpha_i \left( b_1x^{k_1} + b_2x^{k_2} +\cdots +b_sx^{k_s} \right) \right)$, where $b_l\in\F_2$ are arbitrary, as well as $\alpha_1,\alpha_2,\alpha_3,\alpha_4\in\F_{2^{l_1}}$. 

\noindent\emph{2)} More generally, it is enough to take $k$ in such  a way, that $\alpha_1^{k+1}+\alpha_2^{k+1}+\alpha_3^{k+1}+\alpha_4^{k+1}=0$ for a flat $\{\alpha_1,\alpha_2,\alpha_3,\alpha_4\}$ of $\F_{2^m}$, i.e., $\alpha_1+\alpha_2+\alpha_3+\alpha_4=0$ with $\alpha_1^{d+1}=\alpha_2^{d+1}=\alpha_3^{d+1}=\alpha_4^{d+1}=1$ as in Theorem~\ref{th:permutA_m2nd class}. The latter means that $x\mapsto x^{k+1}$ maps the flat $\{\alpha_1,\alpha_2,\alpha_3,\alpha_4\}\subset\F_{2^m}$ to the flat $\{\alpha_1^{k+1},\alpha_2^{k+1},\alpha_3^{k+1},\alpha_4^{k+1}\}\subset\F_{2^m}$. For more details on this topic, we refer to~\cite{vanishingflats}.
\end{remark}

\subsection{Vectorial bent-negabent constructions in comparison to $\Z_{2^k}$-bent functions}

In \cite{kppp}, one of the key ideas employed for constructions of vectorial bent-negabent functions was to use Maiorana-McFarland bent functions with complete permutations combined with the multiplication in certain finite fields to obtain vectorial bent functions, and then find an appropriate affine transformations to transform the functions into ones whose component functions are also negabent.
The following result illustrates certain difficulties which may occur when using the same methods as in \cite{kppp} to construct gbent functions $f \colon  \mathbb{V}_n \to \mathbb{Z}_{2^k}$ for larger $k$, using component functions of vectorial bent-negabent functions with linear permutations.

\begin{proposition}\label{ex: negabent}
Let $\lbrace \alpha_0, \alpha_1, \ldots , \alpha_k \rbrace \subset \F_{2^m}$ be a set of $k+1$ linearly independent elements (over $\F_2$), where $k>2$ if $m$ is odd and $k>3$ if $m$ is even. For $i=0,1, \dots, k$, let $f_i \colon \F_{2^{m}} \times \F_{2^{m}} \to \F_2$ be the function defined by 
$$f_i(x,y)= \Tr^m_1(x \pi_i(y)) + h_i(y), \text{ for all } x,y \in \F_{2^m},$$ where $\pi_i(y)=\alpha_i y$, for all $y \in \F_{2^m}$, and $h_i \colon \F_{2^m} \to \F_2$ are arbitrary. Then, the function $F \colon \F_{2^{m}} \times \F_{2^{m}} \to \Z_{2^{k+1}}$ defined by
$$F(x,y)=f_0(x,y)+2f_1(x,y)+ \ldots 2^{k}f_k(x,y),  \text{ for all } x,y \in \F_{2^m},$$
is not $gbent$.
\end{proposition}
\begin{proof}
By Proposition \ref{cor:dualiff}, the function $F$ is gbent if and only if $ \mathcal{A} = f_{k} + \langle f_0,f_1,\ldots,f_{k-1}\rangle $
is an affine vector space of bent functions such that for any $g_0,g_1,g_2,g_3 \in \mathcal{A}$ with $g_0+g_1+g_2+g_3 = 0$ we have
$g_0^*+g_1^*+g_2^*+g_3^* = 0$. Since
$\alpha_0, \alpha_1, \ldots , \alpha_k$ are linearly independent,
$ \mathcal{A} = f_{k} + \langle f_0,f_1,\ldots,f_{k-1}\rangle $
is indeed an affine vector space of bent functions in the Maiorana-McFarland class. Set $g_0=f_k+f_0$, $g_1=f_k+f_1$, $g_2=f_k+f_2$ and $g_3=f_k+f_0+f_1+f_2$. We then have $g_0+g_1+g_2+g_3 = 0$. The dual of the bent function of the form $f(x,y)=\Tr^m_1(x \pi(y)) + h(y)$ is $f^*(x,y)=\Tr^m_1(y \pi^{-1}(x)) + h(\pi^{-1}(x))$. Hence, if we set $\beta_i=\alpha_i+\alpha_k$, for $i=0,1,2$, and $\beta_3=\alpha_0+\alpha_1+\alpha_2+\alpha_k$, we deduce that 
$(g_0^*+g_1^*+g_2^*+g_3^*)(x,y)$ is equal to $$\Tr^m_1(y (\beta_0^{-1}+\beta_1^{-1}+\beta_2^{-1}+\beta_3^{-1})x) + h^*(x),$$
for all $x,y \in \F_{2^m}$ and some $h^* \colon \F_{2^m} \to \F_2$. Hence, $g_0^*+g_1^*+g_2^*+g_3^* = 0$ implies $\beta_0^{-1}+\beta_1^{-1}+\beta_2^{-1}+\beta_3^{-1}=0$. However, the multiplicative inverse permutation $\phi \colon \F_{2^m} \to \F_{2^m}$, defined by $\phi(x)=x^{-1}$, for all $x \in \F_{2^m}^*$ and $\phi(0)=0$, is almost perfect nonlinear for odd $m$. Since $\beta_0+\beta_1+\beta_2+\beta_3=0$, the set $\{\beta_0,\beta_1,\beta_2,\beta_3 \}$ is a $2$-dimensional flat, and consequently $\beta_0^{-1}+\beta_1^{-1}+\beta_2^{-1}+\beta_3^{-1} \neq 0$, because $\phi(x)=x^{-1}$ is APN. We conclude that the function $F$ is not gbent for odd $m$ and $k>2$. 

Assume now that $m$ is even and $k>3$. Using the same notation, set $\overline{g}_2=f_k+f_3$ and $\overline{g}_3=f_k+f_0+f_1+f_3$. Define $\overline{\beta}_2=\alpha_3+\alpha_k$ and $\overline{\beta}_3=\alpha_0+\alpha_1+\alpha_3+\alpha_k$.
Then we have $g_0+g_1+\overline{g}_2+\overline{g}_3 = 0$ and the sum of the duals  
$(g_0^*+g_1^*+\overline{g}_2^*+\overline{g}_3^*)(x,y)$ is equal to $$\Tr^m_1(y (\beta_0^{-1}+\beta_1^{-1}+\overline{\beta}_2^{-1}+\overline{\beta}_3^{-1})x) + \overline{h}^*(x),$$
for all $x,y \in \F_{2^m}$ and some $\overline{h}^* \colon \F_{2^m} \to \F_2$. Hence, $g_0^*+g_1^*+\overline{g}_2^*+\overline{g}_3^*=0$ implies 
$\beta_0^{-1}+\beta_1^{-1}+\overline{\beta}_2^{-1}+\overline{\beta}_3^{-1}=0$. 
As before, let $\phi \colon \F_{2^m} \to \F_{2^m}$ be the multiplicative inverse permutation. We have
\begin{equation*}
	\begin{split}
		\phi(\beta_0)+ \phi(\beta_0+(\beta_1+\beta_0)) &= \phi(\beta_2)+ \phi(\beta_2+(\beta_1+\beta_0))\\   &=\phi(\overline{\beta}_2)+ \phi(\overline{\beta_2}+(\beta_1+\beta_0)).
	\end{split}
\end{equation*}
If we set $\phi(\beta_0)+ \phi(\beta_1)=d$, we deduce that the equation $\phi(x)+\phi(x+(\beta_1+\beta_0))=d$ has at least $6$ solutions. However, this is a contradiction because for even $m$ the permutation $\phi \colon \F_{2^m} \to \F_{2^m}$ is $4$-differentially uniform. We conclude that $F$ can not be gbent for even $m$ and $k>3$. \qed
\end{proof}

On the other hand, the following result shows that it is actually possible to construct $\mathbb{Z}_{2^k}$-bent functions using similar ideas as in Proposition \ref{ex: negabent}. However, since the permutations used in the construction are not complete, it is not possible to transform them and obtain vectorial bent-negabent functions using the methods from \cite{kppp}.

\begin{theorem}\label{ex: inverseZk}
Let $\lbrace \alpha_0, \alpha_1, \ldots , \alpha_{k-1} \rbrace \subset \F_{2^m}$ be a set of $k$ linearly independent elements (over $\F_2$). For $i=0,1, \dots, k-1$, let $f_i \colon \F_{2^{m}} \times \F_{2^{m}} \to \F_2$ be the function defined by 
$$f_i(x,y)= \Tr^m_1(x \pi_i(y)), \text{ for all } x,y \in \F_{2^m},$$ 
where $\pi_i(y)=\alpha_i y^{-1}$, for all $y \in \F_{2^m}^*$ and $\pi_i(0)=0$. Let $F \colon \F_{2^{m}} \times \F_{2^{m}} \to \Z_{2^{k}}$ be the generalized Boolean function defined by
$$F(x,y)=f_0(x,y)+2f_1(x,y)+ \cdots + 2^{k-1}f_{k-1}(x,y),  \text{ for all } x,y \in \F_{2^m}.$$
Then, $F$ is a $\mathbb{Z}_{2^k}$-bent function.
\end{theorem}
\begin{proof} 
Take any $\alpha \in \F_{2^m}^*$, and define $\pi_{\alpha}(y)=\alpha y^{-1}$, for all $y \in \F_{2^m}^*$ and $\pi_{\alpha}(0)=0$.
Note that $\pi_{\alpha}^{-1}(y)=\pi_{\alpha}(y)$, for all $y \in \F_{2^m}^*$, because $\alpha (\alpha y^{-1})^{-1} =y$. Also, we have $\pi_{\alpha}^{-1}(0)=\pi_{\alpha}(0)=0$, hence $\pi_{\alpha}^{-1}=\pi_{\alpha}$. We deduce that the dual of the function $f_{\alpha}(x,y)=\Tr^m_1(x \pi_{\alpha}(y))$ is given by
\begin{equation}\label{eq: dualsofalpha}
	f_{\alpha}^*(x,y)= \Tr^m_1(y \pi_{\alpha}(x))=f_{\alpha}(y,x), \text{ for all } x,y \in \F_{2^m}.
\end{equation}
Combining Equation \eqref{eq: dualsofalpha} with Proposition \ref{cor:dualiff}, we deduce that $F$ is a $\mathbb{Z}_{2^k}$-bent function.
\qed \end{proof}

Motivated by Proposition \ref{ex: negabent} and Theorem \ref{ex: inverseZk}, we propose the following open problem regarding the existence of vectorial bent-negabent functions whose coordinate functions form a $\mathbb{Z}_{2^k}$-bent functions.

\begin{problem}\label{problem: vec negabent and Z2k-bent}
Find vectorial bent-negabent functions $F\colon \F_2^{2m} \to \F_2^{k}$ given by $F=(f_0,f_1, \ldots ,f_{k-1})$ such that the function $F'\colon \F_2^{2m} \to \mathbb{Z}_{2^k}$ defined by $$F'=f_0+2f_1+ \cdots + 2^{k-1}f_{k-1}$$  is a $\mathbb{Z}_{2^k}$-bent function, or show that such functions do not exist.
\end{problem}

In the sense of Problem~\ref{problem: vec negabent and Z2k-bent}, we suggest the following (very related) questions:
\begin{enumerate}
	\item Is it possible to find functions which are simultaneously gbent and nega-gbent?
	\item Is it possible to find functions which are simultaneously $\Z_{2^k}$-bent and nega-$\Z_{2^k}$-bent?
	\item With Maiorana-McFarland bent functions stemming from complete mappings one can derive vectorial bent-negabent functions, as in \cite{kppp}. However,  as Theorem~\ref{ex: inverseZk} illustrates, the coordinate functions of a gbent or even $\Z_{2^k}$-bent function cannot be specified using them. Do there exist combinatorial objects that can be simultaneously identified as vectorial bent-negabent and gbent (alternatively $\Z_{2^k}$-bent) functions?
	
%	The example of a $\Z_{2^k}$-bent function from Theorem~\ref{ex: inverseZk}, which arises from $\mathcal{PS}_{ap}$ class, as well as from the Maiorana-McFarland class (but not from complete mappings), does not correspond to vectorial bent-negabent functions. Can we find something which serves both?
	\item The previous problem leads to an interesting subcase, thus whether there exist complete permutations that satisfy the $(\mathcal{A}_m)$ property? %then the question is whether these permutations can be complete and used therefor to build vectorial bent-negabent functions? 
\end{enumerate}

\section*{Acknowledgements} 
The authors would like to thank Alexander Pott for helpful
suggestions all along this work. \\
Nurdag\"{u}l Anbar is supported by T\"UB\.ITAK Project under Grant 120F309.
Wilfried Meidl is supported by the FWF Project P 35138.
Sadmir Kudin and Enes Pasalic are partly supported  by the Slovenian Research Agency (research program P1-0404 and research projects N1-0159, J1-2451 and J1-4084).

%\noindent The authors would like to thank the anonymous reviewers for their valuable comments, which helped to improve the presentation of the results.

%\bibliographystyle{spmpsci}
%\bibliography{bibitems_MNBC}

\end{document}